%% file: arxiv.tex
\documentclass[12pt,reqno]{amsart}
\calclayout
\usepackage{amsthm,amsmath,amsfonts,amssymb,mathtools,dsfont}
\mathtoolsset{showonlyrefs}
\usepackage{graphicx}
\usepackage[mathcal]{euscript}
\usepackage{xcolor}
\usepackage{tikz}
\usetikzlibrary{cd}
\usepackage{amsrefs}
\usepackage[colorlinks,citecolor=blue,urlcolor=blue, linkcolor=blue]{hyperref}

\theoremstyle{plain}
\newtheorem{theorem}{Theorem}[section]

\newtheorem{lemma}[theorem]{Lemma}
\newtheorem{proposition}[theorem]{Proposition}

\theoremstyle{remark}
\newtheorem{definition}[theorem]{Definition}
\newtheorem{example}[theorem]{Example}
\newtheorem{remark}[theorem]{Remark}

\newtheorem*{assumption*}{Assumption}

\input{newcommands}

\newcommand{\EE}{\mathrm{EE}}
\newcommand{\NE}{\mathrm{NE}}
\newcommand{\EN}{\mathrm{EN}}
\newcommand{\E}{\mathrm{E}}
\newcommand{\N}{\mathrm{N}}
\newcommand{\NN}{\mathrm{NN}}

\newcommand{\mfC}{\mathfrak{C}}

\def\relog{\mathtt{Log}}
\def\sing{\mathcal{V}}

\begin{document}
\title[The Efficiency-Complexity Tradeoff]{EFFICIENCY--COMPLEXITY TRADEOFF AND CONSTRAINED PERCOLATION}

\author{Yuliy Baryshnikov}
\address{Department of Mathematics and ECE, University of Illinois Urbana-Champaign}
\email{ymb@uiuc.edu}

\author{Efe Onaran}
\address{Department of Mathematics and Department of Electrical and Systems Engineering, University of Pennsylvania}
\email{eonaran@seas.upenn.edu}


\begin{abstract}
\input{abstract}
\end{abstract}

\maketitle

\input{main}

\bibliography{references}  
\end{document}

%% file: newcommands.tex
\DeclareMathOperator{\card}{\mathtt{card}}
\DeclareMathOperator{\amoeba}{\mathtt{amoeba}}

\DeclareMathOperator*{\argmin}{arg\,min\;}

\newcommand{\zd}{\mathbb{Z}^d}
\newcommand{\Qd}{\mathbb{Q}^d}
\newcommand{\limas}{\xrightarrow{\textup{a.s.}}}

\newcommand{\seceq}{\setcounter{equation}{0}}

\newcommand{\pr}{\mathbb P}
\newcommand{\ex}{\operatorname{\mathbb E}}

\newcommand{\mc}{\mathcal}
\newcommand{\mf}{\mathfrak}
\newcommand{\mbb}{\mathbb}
\newcommand{\ind}{\operatorname{\mathds{1}}}
\newcommand{\mb}{\mathbf}
\newcommand{\bs}{\boldsymbol}

%% file: abstract.tex
We study the efficiency--complexity tradeoff for resource-limited agents through directed first-passage percolation. Nonnegative i.i.d. edge weights represent physical energy expenditure, while the logarithmic size of a constrained path family measures representational complexity. For translation-invariant families, a path-joining property yields deterministic large-distance costs in rational directions and joint convexity in constraint parameter and direction, under suitable moment and feasibility assumptions. Our main result concerns the limiting cost parameterized by path-family entropy. Using techniques from analytic combinatorics in several variables (ACSV), we establish its almost-sure existence and joint continuity in entropy and direction for nested families satisfying additional generating-function conditions, on the corresponding entropy range. Path-count asymptotics supply the change of parameter linking this cost to the constraint-parameterized cost. Two planar examples constrain the number of turns or the number of East-runs of length at least two.

%% file: main.tex
\section{Introduction}
\subsection{Motivation}
Biological organisms routinely operate under simultaneous constraints on physical affordances, such as the energy budget, maximal available force etc., and information-processing capacity, be that computational bandwidth, or available memory. Consequently, it is natural to expect a trade-off between the energetic cost of physical operations and the costs associated with acquiring, storing, and processing information. This trade-off appears across scales, from intracellular regulatory networks and bacterial chemotaxis to animal navigation, collective behavior, and cognitive decision making.

Generically, we will be referring to this as the {\em efficiency-complexity trade-off}.

The idea that optimal behavior in a natural system must be understood relative to finite computational resources has a long history. Simon's theory of bounded rationality challenged the assumption that biological or economic agents can perform unconstrained optimization, arguing instead that decision making is fundamentally limited by memory, computation, and available information \cites{Simon1955,Simon1956}. More recently, information-theoretic approaches to bounded rationality have formulated these limitations by introducing explicit costs associated with information processing \cites{OrtegaBraun2013,OrtegaEtAl2015}. In these formulations, an agent's objective is not merely to maximize utility but rather to optimize a balance between utility and the information required to achieve it.

Related perspectives arise in neuroscience and cognitive science. Resource-rational theories model cognition as the optimal use of limited computational resources \cite{LiederGriffiths2020}. Rate-distortion  describes the trade-off between representational complexity and behavioral performance using the tools of information theory \cites{Sims2016,ArumugamEtAl2022}. In evolutionary settings, organisms may benefit from maintaining richer internal representations of their environments, but these representations require metabolic resources to develop and maintain. The resulting compromises between representational fidelity and physical costs have been studied through theories of lossy biological compression and efficient coding \cites{MarzenDeDeo2017,Barlow1961}.

In biological systems these information processing constraints manifest themselves through substantial energetic demands of nervous systems. Formation of neural tissue is metabolically expensive, and information processing itself consumes a significant fraction of an organism's energy budget. Experimental and theoretical studies have argued that energetic constraints play a central role in shaping neural coding, sensory processing, and cognitive architecture \cites{LaughlinEtAl1998,NivenLaughlin2008,JamadarEtAl2024}. From this perspective, an organism cannot arbitrarily increase the sophistication of its physical actions without incurring a corresponding energetic cost.

Navigation provides a natural setting in which these competing pressures interact. Consider an organism attempting to move through a heterogeneous environment. Access to a larger set of candidate routes generally improves the quality of the chosen trajectory, reducing the required travel time or energetic expenditure. However, evaluating, storing, or representing a larger collection of possible paths requires greater computational resources. Similar considerations arise in animal migration, foraging, insect route learning, and spatial cognition, where improved navigation performance is believed to rely on increasingly sophisticated internal representations of the environment \cites{Tolman1948,OkeefeNadel1978}. An analogous trade-off appears in autonomous robotics and motion planning, where computational effort is often exchanged for improved path quality.

\subsection{Goals}
The present paper develops a mathematical framework for studying this general principle in the setting of directed first-passage percolation. Classical first-passage percolation seeks the minimum-weight path through a random environment and may be viewed as an idealized agent capable of evaluating all admissible routes. In contrast, we consider agents whose search is restricted to a constrained family of paths. These constraints represent limitations on memory, computational power, or representational complexity. The resulting optimization problem captures a trade-off between physical efficiency and computational affordance: enlarging the set of admissible trajectories improves the quality of the best available route, but simultaneously increases the informational resources required to describe and search that set.

A central theme of this work is the interpretation of path-family size as a measure of computational complexity. If the admissible path family contains $N$ candidate trajectories, then approximately ($\log_2 N$) bits are required to specify an element of that family. This observation motivates the entropy-based parameterization introduced later in the paper. From this viewpoint, the growth rate of the admissible path class quantifies the information-processing capability of the agent, while the associated first-passage time quantifies its physical performance. The resulting framework can therefore be interpreted as a probabilistic model of the trade-off between energetic expenditure in the environment and energetic expenditure devoted to information processing.

\subsection{Outline of Results}
Our main results concern filtrations of directed lattice paths: translation-invariant path classes in which joining two permissible paths gives a permissible path at the length-weighted average of their constraint parameters. The constraint parameter limits path complexity, for example by bounding the number of turns per unit length. The edge weights represent the physical energy expended in traversing the corresponding edges; we assume these weights are i.i.d., nonnegative, and have finite second moment. Under a feasibility condition requiring eventual nonemptiness of path families along each nonzero nonnegative lattice ray, we prove almost-sure convergence to deterministic time constants in rational directions and joint convexity in the constraint parameter and direction (Lemmas~\ref{subadd} and~\ref{convexitylem}). For nested feasible filtrations, in which increasing the parameter retains every previously permissible path, we also prove that the exponential growth rate of the permissible path counts exists and is jointly continuous and concave (Lemma~\ref{limlem}). When the path-count generating function is rational, we express this rate as the infimum of a linear function over its logarithmic convergence domain; under the assumptions of Definition~\ref{defsmoothpathclass}, the rate can be evaluated at positive minimal critical points. These results yield an entropy-based reparameterization for which the time constant exists and is jointly continuous in entropy and direction on the corresponding entropy range (Theorem~\ref{reparamconv}). We verify the required assumptions for two planar examples, constraining the number of turns or the number of East-runs of length at least two (Sections~\ref{turnssmooth} and~\ref{2Es}).

\section{A Toy Model}
\seceq

As an illustration of the \emph{complexity vs. efficiency} tradeoff, we consider the following toy example. 

Consider an agent whose goal is to penetrate a \emph{layered} environment\footnote{In the spirit of time, one can envision an agentic exploration with the goal of finding vulnerabilities in a deeply protected system.}. To reach the $n$th layer starting from layer-0, the agent executes steps, with some options to follow at each step: the (perhaps stochastic) number of such options depends on the resources ({\em computational}, e.g., how many options at each step it can process) the agent is afforded. Each step takes some time to execute, and the agent's task is to find the most expeditious ({\em efficient}) path to the $n$-th level.
We assume, for simplicity, that the durations of each option at each step are positive {\em i.i.d.} random variables.

The optimization problem can be written as
\begin{align}
    M_n\coloneqq \min_{|\mb P|=n}\sum_{e\in \mb P} \tau_e,
\end{align}
where the minimum is over all paths $\mb P$ from layer-0 to layer-$n$, and $\tau_e$ denotes the i.i.d. duration associated with edge $e$.

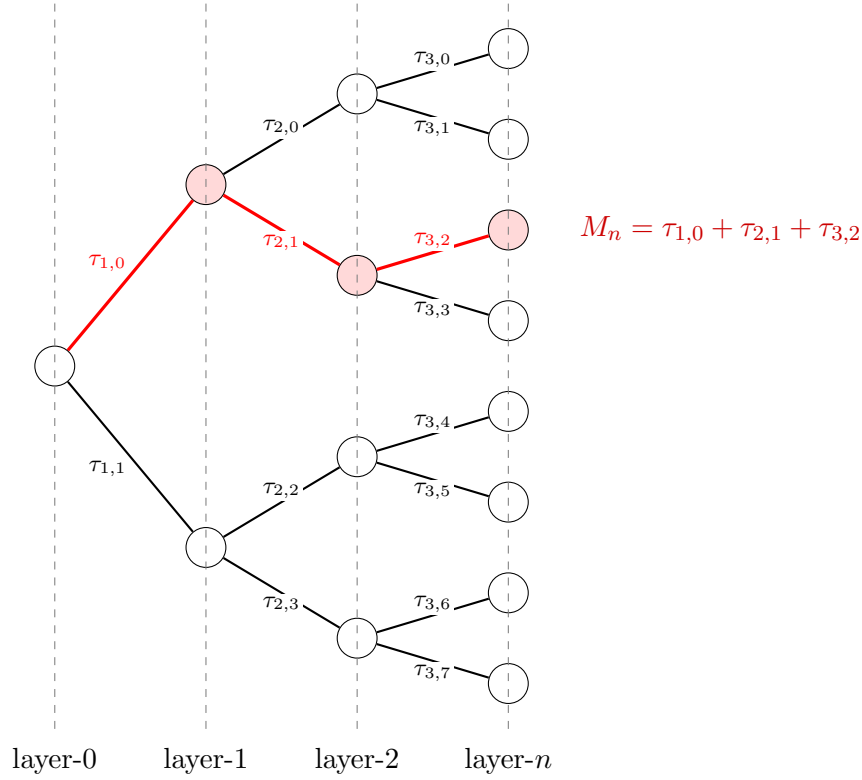
\begin{figure}[ht] 
\centering
\input{tikz/BranchRW}
\caption{Depiction of a toy model that illustrates the cost efficiency vs. computational overhead tradeoff. An agent is aiming to reach layer-$n=3$ from layer-0 by traveling through the edges of a tree where each node has $b=2$ children, and each edge is assigned a non-negative i.i.d. additive cost which we denote by $\tau_{i,j}$ for the $j$'th edge between layer-$(i-1)$ and layer-$i$. The (assumed) path with the least cost is shown in red and the associated cost is $M_n$.}  
\label{BranchRWfig}
\end{figure}

One recognizes, of course, the well-studied problem on \emph{branching random walk}, where, in general, the number of children of each node is a random i.i.d. integer variable $>1$ (to assure survival) with bounded mean number $b$ of offspring, so that the environment is a realization of the \emph{Galton-Watson tree}.

If one assigns {\em iid} random duration $\tau_e$ to each edge, we have the question of what is the shortest path from the root to the level $n$. It had been established in \emph{Biggins–Hammersley–Kingman} theorem \cites{biggins, hammersley, kingmanbranch}
that 
\begin{align}
    \lim_{n\to\infty} \frac{M_n}{n} = \mu
\end{align}
almost surely, for a deterministic $\mu$ which depends on $b$ and the distribution of the edge durations $\tau_e\sim X$. The standard heuristic (computing the expected number of paths to the level $n$ of the length $\leq n\mu$) leads to the following description of $\mu$.

Let $\Lambda(u)\coloneqq \log \ex[e^{u\tau}]$ be the cumulant generating function for the (non-negative) random variable $\tau$, and $u_*<0$ solves
\[u\Lambda'(u)-\Lambda(u) = \log b. \]
Then
\[\mu= \Lambda'(u_*).\] 

One can easily see that as $b$ grows, $\mu$ decreases. This matches the intuition: the more computational resources the agents are given, the better solutions they come up with.

Consider an example where each edge cost is {\em exponentially distributed with unit mean}. Then $\Lambda(u)=-\log(1-u)$, whence $u_*$ satisfies
\[\frac{u_*}{1-u_*}+\log(1-u_*) = \log b, \]
and (from $\Lambda'(u_*)=(1-u_*)^{-1}=\mu$)
\begin{align}
    \mu-1 - \log \mu = \log b.
\end{align}
It is straightforward to check that $\mu$ can be expanded into asymptotic series in $b^{-1}$, with the leading term
\[
\mu=\frac{1}{eb}+O(b^{-2})
\] 
giving a clear expression of the computational affordance (agent's abilities) vs. time to success (efficiency) tradeoff.

\section{Main Model}
\seceq
Unlike the model in the previous section, our main efforts and results are focussed on the percolation on lattices $\zd$.

\subsection{Percolation}
Assume that each edge between the nearest neighbors of the lattice $\zd$, $d\geq 2$ is assigned an i.i.d. nonnegative random variable $\tau_{\bs x,\bs y}$ for each edge $(\bs x,\bs y), \bs x,\bs y\in \zd$ such that $|\bs x-\bs y|=1$ (here $|\cdot|$ denotes the $\ell_1$-norm). We will denote by $\tau$ an independent random variable with the common distribution of the weight of each edge. We assume throughout that
\begin{align}
    \ex\tau<\infty \quad\text{and}\quad \ex\tau^2<\infty.
\end{align}

We let $\bs o$ denote the origin of $\zd$, and write $\bs x\leq \bs y$ if and only if
 $x_i\leq y_i $ for all $1\leq i\leq d$, where $\bs x = (x_1,\ldots,x_d)$ and $\bs y = (y_1,\ldots,y_d)$. If $\bs x\leq \bs y $ and $\bs x \neq  \bs y $, we denote this by $\bs x<\bs y$.
We define a path $\mb A$ of length-$k\geq 1$, by the sequence of points it passes through, i.e., $(\bs a_0,\bs a_1,\ldots, \bs a_{k})$ where $\bs a_i\in \zd$ for all $0\leq i\leq k$ and  
\begin{align}
|\bs a_{i-1}-\bs a_{i}| &=1,\\
\bs a_{i}&\geq \bs a_{i-1} 
\end{align}
for all $1\leq i\leq k$. We say that such $\mb A=(\bs a_0,\bs a_1,\ldots, \bs a_{k})$ \emph{starts} at $\bs a_0$ and \emph{ends} at $\bs a_k$, and we will let $|\mb A|=k$ denote its length. We define the \emph{passage time} for the path $A$ as
\begin{align}
    T(\mb A) \coloneqq \sum_{i=1}^{|\mb A|} \tau_{\bs a_{i-1},\bs a_{i}}.
\end{align}
 The \emph{path concatenation} for two paths that both start at $\bs o$ is defined as
\begin{align}
    \mb A + \mb B\coloneqq (\bs o,\bs a_1,\ldots, \bs a_{|A|}, \bs a_{|A|} + \bs b_1,\ldots, \bs a_{|A|} +\bs  b_{|B|}) , 
\end{align}
 for 
 \begin{align}
     \mb A&= (\bs o,\bs a_1,\ldots,\bs  a_{|A|})\\
     \mb B&= (\bs o,\bs b_1,\ldots,\bs  b_{|B|}).
 \end{align}
For two sets of paths $\mc A, \mc{B}$,  we also define the \emph{path set product} $\mc A\times \mc B$ as 
\begin{align}
    \mc{A}\times \mc{B} \coloneqq \{\mb A + \mb B:\mb A\in\mc{A},  \mb B\in \mc{B} \}, 
\end{align}
so, in particular, $\mc A\times\varnothing=\varnothing=\varnothing\times\mc B$.
 We will use the shorthand notations $\mbb Z_{\geq 0}$ ($\mbb Z_{> 0}$), $\mbb Q_{\geq 0}$ ($\mbb Q_{> 0}$), and $\mbb R_{\geq 0}$ ($\mbb R_{> 0}$) for non-negative (positive) integers, rationals, and reals, respectively. The notation $\mbb N$ is also used for positive integers, and $\mbb C$ for complex numbers. We will use the notation $\card(\cdot)$ for the cardinalities of finite sets.  

\subsection{Constrained Percolation}
Classical first-passage percolation problem concerns the distance function of $\bs y,\bs z\in\zd$ defined as 
\begin{align}
    \inf_{\mb C} T(\mb C),
\end{align}
 where the infimum is over the set of all paths starting at $\bs y$ and ending at $\bs z$. 

Here, instead, we will be interested in the \emph{constrained} distance function
\begin{align}
T(\mc{C}_{\bs y,\bs z}^\gamma)\coloneqq \inf_{\mb C\in \mc{C}_{\bs y,\bs z}^\gamma} T(\mb C)
\end{align}
where $\mc{C}_{\bs y,\bs z}^\gamma$ is a  given set of \emph{permissible} paths  that start at $\bs y$ and end at $\bs z> \bs y$, parameterized by some $\gamma\in \mbb R_{\geq 0}$. As we are in the \emph{directed} variant of the first passage percolation, throughout the paper, we will restrict ourselves to the paths with start and end points on the non-negative orthant. In other words, $\mc C_{\bs y,\bs z}^\gamma \neq  \varnothing$ will imply $\bs y\in \zd_{\geq 0}$, $\bs z\in \zd_{\geq 0}$ and $\bs z> \bs y$.  
 
We will be working with  \emph{nested families} of paths
\begin{align}
    \mfC^\gamma \coloneqq \{\mc C^\gamma_{\bs y, \bs z},\; \bs y, \bs z\in\zd_{\geq 0}\}, 
\end{align} 
such that
\[
    \mc C^\gamma_{\bs y,\bs z}\subseteq\mc C^{\gamma'}_{\bs y,\bs z}
    \qquad\text{for all }\bs y,\bs z\in\zd_{\geq0}\text{ and }0\leq\gamma\leq\gamma'.
\]

The corresponding collection of these families, a \emph{path class}, is then
\begin{align}
    \mfC \coloneqq \{\mfC^\gamma,\; \gamma\in\mbb R_{\geq 0}\} = \{\mc C^\gamma_{\bs y, \bs z},\; \gamma\in\mbb R_{\geq 0},\; \bs y, \bs z\in\zd_{\geq 0}\}.
\end{align}

\begin{remark}
Conceptually, the constraint that a path belongs to a path family $\mfC^\gamma$ models the restrictions on the {\em computational} affordances of the system: the greater are those, the more paths the system can explore.
\end{remark}

We will be working primarily with translation invariant families (and classes) of paths:

\begin{definition}[Translation Invariance]\label{transinvass}
We say that the family $\mfC^\gamma = \{\mc C^\gamma_{\bs y, \bs z},\; \bs y, \bs z\in\zd_{\geq 0}\}$ is \emph{translation invariant} if for any $\bs y,\bs z\in\zd_{\geq 0}$ satisfying $\bs y\leq\bs z$,
\begin{align}\label{nonorigineqn}
    \mc{C}_{\bs y,\bs z}^\gamma = \bs y+ \mc{C}_{\bs o,\bs z-\bs y}^\gamma\coloneqq \{\bs y+ \mb C: \mb C\in\mc C_{\bs o, \bs z-\bs y}^\gamma\},
\end{align}
where the translation of a path $\mb C=(\bs o, \bs c_1,\ldots, \bs c_{|\mb C|})$ by $\bs y$ is defined as
\begin{align}
    \bs y+ \mb C \coloneqq (\bs y, \bs y+\bs c_1,\ldots, \bs y+\bs c_{|\mb C|}).
\end{align}
For $\bs y\nleq\bs z$, we set $\mc C^\gamma_{\bs y,\bs z}=\varnothing$.

Accordingly, a class $\mfC \coloneqq \{\mfC^\gamma,\; \gamma\in\mbb R_{\geq 0}\}$ is called translation invariant if the family $\mfC^\gamma$ is so for all $\gamma\in\mbb R_{\geq 0}$.

\end{definition}
Note that for a path family that satisfies translation invariance, the paths starting at $\bs o$ determine the whole family. In that regard, to simplify our notation when we talk about translation invariant families, we will refer to $\mc C_{\bs o,\bs z}^\gamma$ also as $\mc C_{\bs z}^\gamma$.
Furthermore, we make the following definition for a class of paths. 

\begin{definition}[Filtration]\label{pathconsassum}
 A translation invariant class $\mfC = \{\mc C^\gamma_{\bs y, \bs z},\; \gamma\in\mbb R_{\geq 0},\; \bs y, \bs z\in\zd_{\geq 0}\}$ is called a \emph{filtration} if for any $\bs y,\bs  z\in\zd_{\geq 0}$ satisfying $\bs y+\bs z\neq\bs o$, and $\alpha, \beta \in \mbb R_{\geq 0}$, 
 \begin{align}\label{pathcons}
\mc C^{\alpha}_{\bs z} \times \mc C^{\beta}_{\bs y}   \subseteq
    \mc C^{\gamma}_{\bs z+\bs y}.
\end{align}
 where
\begin{align*}
     \gamma = \frac{\alpha |\bs z| + \beta |\bs y|}{|\bs z+\bs y|}. 
\end{align*}

\end{definition}

\begin{remark}
    Note that given a filtration $\mfC$, Definition~\ref{pathconsassum} implies for all $\gamma\in\mbb R_{\geq 0}$, $\bs z\in\zd_{\geq 0}$ and $k\in \mbb N$ that
    \begin{align}
\underbrace{\mc C^{\gamma}_{\bs z} \times \mc C^{\gamma}_{\bs z } \times \cdots \times \mc C^{\gamma}_{\bs z}}_{\text{$k$ times}}   \subseteq
    \mc C^{\gamma}_{k\bs z},  
\end{align}
  and therefore, due to the translation invariance and the symmetry of edge weights, 
  \begin{align}\label{pathcons2ineq}
      \ex[T(\mc C^{\gamma}_{k\bs z})]\leq k \ex[T(\mc C^{\gamma}_{\bs z})].
  \end{align}
\end{remark}

Finally, we make the following definition of \emph{feasibility} for a family.  

\begin{definition}\label{nonemptydef}
A translation-invariant family of paths $\mfC^\gamma$ is called \emph{feasible} if, for every $\bs z\in\zd_{\geq0}\setminus\{\bs o\}$, there exists $n_0\in\mbb N$ such that $\mc C^\gamma_{n\bs z}$ is nonempty for all $n\geq n_0$. The class $\mf C = \{\mf C^{\gamma}\}$ is called feasible if $\mfC^\gamma$ is so for all $\gamma>0$.

\end{definition}

\begin{remark}\label{axispathsrmk}
For a feasible family $\mfC^\gamma$ in a filtration, fix $K_\gamma\in\mbb N$ such that $\mc C^\gamma_{n\bs e_i}\ne\varnothing$ for all $n\geq K_\gamma$ and $i=1,\ldots,d$, where $\bs e_i$ is the $i$th coordinate unit vector. Concatenating these axis paths shows that $\mc C^\gamma_{\bs z}\ne\varnothing$ whenever $\bs z\ne\bs o$ and every coordinate of $\bs z$ is either zero or at least $K_\gamma$.
\end{remark}

\begin{definition}\label{countrmk}
    Take an arbitrary {\em pattern counting} function $f$ with the domain as the set of all paths with the starting point $\bs o$ and range as the non-negative integers, which satisfies the \emph{(almost) subadditivity} property, i.e.,
    \begin{align}\label{fsubadd}
        f(\mb A+\mb B)\leq f(\mb A) + f(\mb B) + 1
    \end{align}
    for all paths $\mb A,\mb B$. For $\gamma \in\mbb R_{\geq 0}$ and $\bs z\in\zd_{\geq 0}$ define
    \begin{align} \label{Czgammadefeqn}
        \mc C_{\bs z}^\gamma \coloneqq \{ \text{$\mb A$ is a path from $\bs o$ to $\bs z$ s.t. $f(\mb A)< \lfloor|\bs z|\gamma\rfloor$} \}.
    \end{align}
     Given $\mc C_{\bs z}^\gamma$, the family $\mfC^\gamma$ is defined through translation invariance.
     It is straightforward to show that the family of permitted sets defined in such way is always a filtration, whereas the feasibility depends on $f$. 
\end{definition}

\subsection{Examples}
The following examples will be explored in detail.
\begin{example}\label{numoftruns exa}
For a path $\mb A=(\bs o,\bs a_1, \ldots,\bs a_{|A|})$, let $f(\mb A)$ be defined as
\begin{align}
      f(\mb A) \coloneqq
          \sum_{i=1}^{|\mb A|-1} \ind\left\{\bs a_{i} \neq \frac{\bs a_{i+1} + \bs a_{i-1}}{2}\right\}.
\end{align}
That is, $f$ counts the number of \emph{turns} on a given path $\mb A$. 
 Then $f$ satisfies \eqref{fsubadd} since there could be at most one additional turn on the concatenation point of two paths. Moreover,
$\mfC^\gamma$ defined through such $f$ is feasible for all $\gamma>0$, therefore $\mf C$ is feasible.

It is instructive to interpret the families $\mfC^\gamma$ from the computational viewpoint. The number of directed paths in $\zd$ of length $n$ starting at $\bs o$, with unrestricted endpoint, is $d^n$, and a typical path has $\approx \frac{d-1}{d} n$ turns. Restricting the path to have no more than $\gamma n$ turns makes this number exponentially smaller than $d^n$, for $0<\gamma<\frac{d-1}{d}$, relaxing the computational requirements (say, number of bits necessary to encode a path). For fixed endpoints $\bs o,\bs z$, $T(\mc C^\gamma_{\bs z})$ measures physical cost, whereas $\log_2|\mc C^\gamma_{\bs z}|$ measures representational complexity.

We will more explicitly treat the complexity in terms of set-size in Section~\ref{reparam sec}.

\end{example}

\begin{remark}
    Another example satisfying \eqref{fsubadd} is the number of maximal runs of a specified step direction of length at least $\ell\geq2$. A maximal run is a consecutive block of steps in that direction that cannot be extended at either end. For example, writing $E$ and $N$ for East and North steps, the path $EEE\,N\,EE\,N\,EEEE$ has three maximal East-runs of length at least two. At the concatenation point, only the final run of the first path and the initial run of the second can merge, so the count is at most the sum of the two counts plus one.
\end{remark}

\section{Limits for Filtrations}
\seceq

Our first result concerns the random variable $T(\mc C^\gamma_{\bs q})$, the weight of the optimal permitted  path from $\bs o$ to a real end point $\bs q\in\mbb R^d_{\geq 0}$, which is defined as 
\begin{align}\label{rationalendptseqn}
T(\mc C_{\bs q}^\gamma) \coloneqq T(\mc C^\gamma_{ \lfloor \bs q\rceil }) = \inf_{\mb A\in \mc C_{\lfloor \bs q\rceil}^\gamma} T(\mb A)
\end{align}
where 
\[\lfloor \bs q\rceil \coloneqq  \text{the lexicographically smallest element of }\argmin_{\bs z\in\zd_{\geq 0}} |\bs q-\bs z|.\]
In particular, this convention is compatible with translation by lattice points: $\lfloor \bs q+\bs z\rceil=\lfloor \bs q\rceil+\bs z$ for all $\bs q\in\mbb R^d_{\geq 0}$ and $\bs z\in\zd_{\geq 0}$.

\subsection{Subadditive Ergodic Theorem}
\begin{lemma}\label{subadd}
Let $\mfC = \{\mfC^\alpha\} = \{\mc C^\alpha_{\bs y, \bs z}\}$ be a filtration, and $\mfC^\gamma$, $\gamma\in\mbb R_{\geq 0}$ be a feasible family in this class. Then, for any $\bs q\in\Qd_{\geq 0}$ satisfying $|\bs q|=1$,
   \begin{align}\label{mucgamqdefeqn}
     \frac{T(\mc C^\gamma_{n\bs q})}{n} \limas \mu_{\mfC}(\gamma,\bs q)
    \coloneqq\inf_{k\geq 1} \frac{\ex [T(\mc C^\gamma_{kN\bs q})]}{kN},
    \end{align}
    where
    \begin{align} \label{Ndefin}
        N \coloneqq&   K_{\gamma} \times \inf\{m\in \mbb N : m\bs q\in\mbb Z_{\geq 0}^d\},
    \end{align}
and $K_\gamma$ is as chosen in Remark~\ref{axispathsrmk}.
\end{lemma}

\begin{proof}
Liggett's sub-additive ergodic theorem \cite{Ligg} asserts that if a sequence of non-negative random variables 
$(X_{m,n})_{0\leq m< n}$ satisfies the following conditions
\begin{enumerate}
    \item $X_{0,n}\leq X_{0,m} + X_{m,n}$, for all $0<m<n$,
    \item $\ex[X_{0,1}]<\infty$,
    \item The distributions of the sequences $(X_{m,m+k})_{k\geq 1}$ and \\ $(X_{m+1,m+k+1})_{k\geq 1}$ are the same for all $m\geq 0$,
    \item For each $k\geq 1 $, the sequence $(X_{nk,(n+1)k})_{n\geq 0}$ is stationary and ergodic,
\end{enumerate}
then 
\begin{align}
 \frac{X_{0,n}}{n}  \limas \inf_{n} \frac{\ex[X_{0,n}]}{n}. 
\end{align}
Using the sub-additive ergodic theorem we first show that 
\begin{align}\label{Tcnn}
  \frac{T(\mc C^\gamma_{kN\bs q})}{kN} \limas \inf_{k} \frac{\ex [T(\mc C^\gamma_{kN\bs q})]}{kN}, 
\end{align}
where $N$ is as defined in \eqref{Ndefin}. The proof of \eqref{Tcnn} follows by defining $X_{m,n}(\bs q) \coloneqq T(\mc C^\gamma_{m\bs z,n\bs z})$, with $\bs z=N\bs q\in\zd_{\geq 0}$, and checking the conditions of Liggett's theorem given above. Each positive coordinate of $\bs z$ is at least $K_\gamma$, so Remark~\ref{axispathsrmk} gives $\mc C^\gamma_{\bs z}\ne\varnothing$. Conditions-3 and 4 follow due to the symmetry and the translation invariance of $\mfC$, while Condition-2 follows from
\begin{align}
    \ex[X_{0,1}] = \ex [T(\mc C^\gamma_{\bs z})] \leq \ex[\tau] \min_{\mb C\in \mc C^\gamma_{\bs z}}|\mb C|<\infty.
\end{align}
Furthermore, Condition-1 follows by noting that 
\begin{align}
    X_{0,n}(\bs q) = T(\mc C^\gamma_{n\bs z}) \leq T(\mc C^\gamma_{m\bs z}) + T(\mc C^\gamma_{m\bs z,n\bs z})  = X_{0,m}(\bs q) + X_{m,n}(\bs q)
\end{align}
since 
\begin{align}
    \mc C_{m\bs z}^\gamma \times \mc C^\gamma_{(n-m)\bs z}
    \subseteq
    \mc C_{n\bs z}^\gamma  
\end{align}
due to \eqref{pathcons}, and $\mc C^\gamma_{m\bs z,n\bs z} = m\bs z + \mc C^\gamma_{(n-m)\bs z} $. Now note that, if $n=kN+j$ for $0\leq j\leq N-1$, $k,j\in \mbb Z_{\geq 0}$ and $n$ is large enough so that $k\geq 2$, then $0\leq\lfloor j\bs q\rceil\leq\bs z$. Since $\lfloor n\bs q\rceil=k\bs z+\lfloor j\bs q\rceil$, the two connecting displacements are
\begin{align*}
\lfloor n\bs q\rceil-(k-1)\bs z&=\bs z+\lfloor j\bs q\rceil,\\
(k+2)\bs z-\lfloor n\bs q\rceil&=2\bs z-\lfloor j\bs q\rceil.
\end{align*}
Both displacements are nonzero. When $z_i=0$, both have zero $i$th coordinate. When $z_i>0$, both have $i$th coordinate at least $z_i$, and $z_i\geq K_\gamma$ by the choice of $N$. Remark~\ref{axispathsrmk} therefore supplies permitted paths from $\bs o$ with these displacements. Translating them to start at $(k-1)\bs z$ and $\lfloor n\bs q\rceil$, respectively, shows that $\mc C^\gamma_{(k-1)\bs z,\lfloor n\bs q\rceil}$ and $\mc C^\gamma_{\lfloor n\bs q\rceil,(k+2)\bs z}$ are nonempty. We have
\begin{align}\label{nqkzineq}
    T(\mc C^\gamma_{n\bs q}) - T(\mc C^\gamma_{(k-1)\bs z}) \leq T(\mc C^\gamma_{(k-1)\bs z,  \lfloor n\bs q\rceil }),
\end{align}
where we used again the definitions of translation invariance and filtration. 
Similarly,
\begin{align}
    T(\mc C^\gamma_{(k+2)\bs z}) - T(\mc C^\gamma_{n\bs q}) \leq  T(\mc C^\gamma_{\lfloor n\bs q\rceil,(k+2)\bs z}),
\end{align}
which, together with \eqref{nqkzineq}, implies that 
\begin{align} 
\begin{alignedat}{2}
\frac{\left|T(\mc C^\gamma_{n\bs q}) - T(\mc C^\gamma_{k\bs z})\right|}{n} \leq & \frac{\left|T(\mc C^\gamma_{(k+2)\bs z}) - T(\mc C^\gamma_{k\bs z})  \right|}{n} + \frac{\left|T(\mc C^\gamma_{k\bs z}) - T(\mc C^\gamma_{(k-1)\bs z})  \right|}{n} \\  &+\frac{T(\mc C^\gamma_{\lfloor n\bs q\rceil, (k+2)\bs z})}{n} + \frac{T(\mc C^\gamma_{(k-1)\bs z, \lfloor n\bs q\rceil })}{n}
\end{alignedat}
\end{align}
 Note that,
\begin{align}
\frac{\left|T(\mc C^\gamma_{(k+2)\bs z}) - T(\mc C^\gamma_{(k-1)\bs z})  \right| + \left|T(\mc C^\gamma_{k\bs z}) - T(\mc C^\gamma_{(k-1)\bs z})  \right|}{n}\limas 0
\end{align}
due to \eqref{Tcnn}. Additionally,
\begin{align}
    \frac{T(\mc C^\gamma_{(k-1)\bs z, \lfloor n\bs q\rceil })}{n} \limas 0,
\end{align}
using Borel-Cantelli, since
\begin{align}
    \sum_{k=1}^\infty \pr\left[T(\mc C^\gamma_{(k-1)\bs z, \lfloor n\bs q\rceil })>n\epsilon\right]&\leq \sum_{k=1}^\infty \frac{\ex \left(T(\mc C^\gamma_{(k-1)\bs z, \lfloor n\bs q\rceil })\right)^2}{n^2\epsilon^2}\\
    &\leq \sum_{k=1}^\infty \frac{\ex \left(T(\mc C^\gamma_{(k-1)\bs z, \lfloor n\bs q\rceil})\right)^2}{k^2\epsilon^2} <\infty
\end{align}
for any $\epsilon>0$ and $j\in\{0,\ldots,N-1\}$. The last inequality above follows from
\begin{align}
    \ex \left(T(\mc C^\gamma_{(k-1)\bs z, \lfloor n\bs q\rceil})\right)^2&\leq \ex[\tau^2] \min_{\mb C\in \mc C^\gamma_{\lfloor n\bs q\rceil -(k-1)\bs z}}|\mb C|^2 \\&= \ex[\tau^2] \min_{\mb C\in \mc C^\gamma_{\bs z+\lfloor j\bs q\rceil}}|\mb C|^2  
    <\infty,
\end{align}
due to the homogeneity of edge weights and the nonemptiness of $\mc C^\gamma_{\bs z+\lfloor j\bs q\rceil}$ established above. Similarly,
\begin{align}
    \frac{T(\mc C^\gamma_{\lfloor n\bs q\rceil, (k+2)\bs z })}{n}\limas 0
\end{align}
follows from 
\begin{align*}
    \ex \left(T(\mc C^\gamma_{\lfloor n\bs q\rceil, (k+2)\bs z })\right)^2&\leq \ex[\tau^2] \min_{\mb C\in \mc C^\gamma_{(k+2)\bs z-\lfloor n\bs q\rceil}}|\mb C|^2 \\&= \ex[\tau^2] \min_{\mb C\in \mc C^\gamma_{2\bs z-\lfloor j\bs q\rceil}}|\mb C|^2<\infty.
\end{align*}
\end{proof}

\begin{remark}
Accordingly, from now on we consider the time constant $\mu_{\mfC}(\gamma,\bs q)$ for rational directions $\bs q$ satisfying $|\bs q|=1$.
\end{remark}

\subsection{Convexity}
Now that we have proved the existence of time constant, we next state a convexity statement thereof.
\begin{lemma}\label{convexitylem}
Let $\mfC = \{\mfC^\alpha\}$ be a filtration. Then the time constant  $\mu_{\mfC}(\gamma,\bs q)$, as defined in \eqref{mucgamqdefeqn},  is ``jointly'' convex in $\gamma$ and $\bs q$, i.e., for any $\bs{q}_1, \bs{q}_2\in\Qd_{\geq 0}\setminus \bs o$, $|\bs q_1|=|\bs q_2|=1$, $t_1,t_2\in \mbb Q_{\geq 0}$, $t_1+t_2 = 1$, and $\alpha, \beta\in \mbb R_{\geq 0}$,   
    \begin{align}
        \mu_{\mfC}(t_1\alpha + t_2\beta,t_1 \bs{q}_1 + t_2 \bs{q}_2) \leq t_1\mu_{\mfC}(\alpha,\bs{q}_1) + t_2\mu_{\mfC}(\beta,\bs{q}_2)  
    \end{align}
    for all $\mfC^\gamma$, $\mfC^\alpha$, $\mfC^\beta$ feasible families in $\mfC$.
\end{lemma}
\begin{proof}
Set
\[
    \gamma\coloneqq t_1\alpha+t_2\beta,
    \qquad
    \bs q\coloneqq t_1\bs q_1+t_2\bs q_2.
\]
If $t_1=0$ or $t_2=0$, the assertion is immediate, so suppose that
$t_1,t_2>0$.  For each feasible parameter--direction pair $(\eta,\bs r)$, write
\[
    N_{\eta,\bs r}
    \coloneqq K_\eta\inf\{m\in\mbb N:m\bs r\in\mbb Z^d_{\geq0}\},
\]
as in \eqref{Ndefin}.  Since $t_1,t_2,\bs q_1,\bs q_2$ are rational, we may
choose $D\in\mbb N$ such that
\begin{align}\label{commonblockscale}
    \frac{D}{N_{\gamma,\bs q}},\qquad
    \frac{Dt_1}{N_{\alpha,\bs q_1}},\qquad
    \frac{Dt_2}{N_{\beta,\bs q_2}}
    \quad\text{are all positive integers}.
\end{align}

We first record the convergence of the normalized expected costs along these
block scales.  For a feasible pair $(\eta,\bs r)$, define
\[
    a_k^{\eta,\bs r}
    \coloneqq \ex[T(\mc C^\eta_{kN_{\eta,\bs r}\bs r})],
    \qquad k\geq1.
\]
The filtration property \eqref{pathcons} and translation invariance imply
\[
    a_{k+\ell}^{\eta,\bs r}
    \leq a_k^{\eta,\bs r}+a_\ell^{\eta,\bs r}.
\]
Hence, by Fekete's lemma and the definition in \eqref{mucgamqdefeqn},
\begin{align}\label{expectedblocklimit}
    \frac{a_k^{\eta,\bs r}}{kN_{\eta,\bs r}}
    \longrightarrow \inf_{\ell\geq1}
    \frac{a_\ell^{\eta,\bs r}}{\ell N_{\eta,\bs r}}
    =\mu_{\mf C}(\eta,\bs r).
\end{align}

For $n\geq1$, apply \eqref{pathcons} with
$\bs y=nDt_1\bs q_1$ and $\bs z=nDt_2\bs q_2$.  This gives
\[
    T(\mc C^\gamma_{nD\bs q})
    \leq T(\mc C^\alpha_{nDt_1\bs q_1})
    +T(\mc C^\beta_{nDt_1\bs q_1,nD\bs q}).
\]
By translation invariance, the last term has the same distribution as
$T(\mc C^\beta_{nDt_2\bs q_2})$.  Taking expectations and dividing by $nD$
therefore yields
\begin{align}\label{commonblockineq}
    \frac{\ex[T(\mc C^\gamma_{nD\bs q})]}{nD}
    \leq t_1\frac{\ex[T(\mc C^\alpha_{nDt_1\bs q_1})]}{nDt_1}
    +t_2\frac{\ex[T(\mc C^\beta_{nDt_2\bs q_2})]}{nDt_2}.
\end{align}

By \eqref{commonblockscale}, each normalized expectation in
\eqref{commonblockineq} is a subsequence of the corresponding sequence in
\eqref{expectedblocklimit}.  Consequently, as $n\to\infty$, the three terms
converge respectively to
\[
    \mu_{\mf C}(\gamma,\bs q),\qquad
    \mu_{\mf C}(\alpha,\bs q_1),\qquad
    \mu_{\mf C}(\beta,\bs q_2).
\]
Taking the limit in \eqref{commonblockineq} gives
\[
    \mu_{\mf C}(\gamma,\bs q)
    \leq t_1\mu_{\mf C}(\alpha,\bs q_1)
    +t_2\mu_{\mf C}(\beta,\bs q_2),
\]
as required.
\end{proof}

\begin{remark}
For a given feasible filtration $\mfC = \{\mfC^\alpha\}$, the restriction of $\mu_{\mf C}$ to
\[
    \mbb R_{>0}\times \{\bs q\in \mbb Q_{>0}^d: |\bs q|=1\}
\]
has a unique continuous extension to
\[
    \mbb R_{>0}\times \{\bs q\in \mbb R_{>0}^d: |\bs q|=1\}.
\]
Continuity follows from the standard local estimate for convex functions, using Lemma~\ref{convexitylem} and the bound $0\leq\mu_{\mf C}\leq\ex\tau$. On this domain, the convexity inequality in Lemma~\ref{convexitylem} also holds for $t_1,t_2\in \mbb R_{\geq 0}$, $t_1+t_2=1$.
\end{remark}

\begin{remark}
    On this domain, taking $\bs{q}_1=\bs{q}_2=\bs q$, we recover the convexity of $\mu_{\mfC}(\gamma,\bs q)$ in $\gamma$, that is,
       \begin{align}\label{convmu}
        \mu_{\mfC}(\gamma,\bs q) \leq t\mu_{\mfC}(\alpha,\bs q) + (1-t)\mu_{\mfC}(\beta,\bs q),
    \end{align}
for all $0<t<1$, where $\gamma = t\alpha + (1-t)\beta$.
    We can also choose $\gamma=\alpha=\beta$ in Lemma~\ref{convexitylem} to obtain the convexity of $\mu_{\mfC}(\gamma,\bs q)$ in $\bs q$, i.e.,
    \begin{align}
        \mu_{\mfC}(\gamma,t\bs q_1+(1-t)\bs{q}_2) \leq t\mu_{\mfC}(\gamma,\bs{q}_1) + (1-t)\mu_{\mfC}(\gamma,\bs{q}_2).
    \end{align}
\end{remark}

\section{Reparameterization by the class size}\label{reparam sec}
\seceq
Up to now we were considering arbitrary (nested) families of paths, parameterized by some auxiliary parameter (having sometimes a clear meaning in terms of the underlying problem, like the number of turns).

\subsection{Entropic Parameterization}
In this section, we consider a rather more natural parameterization of the permissible paths according to {\em the size of the families of a filtration $\mf{C}$}. 

\begin{definition}\label{entropdef}
Let $\mfC = \{\mfC^\alpha\} = \{\mc C^\alpha_{\bs z}\}$ be a translation invariant path class. Define the \emph{inverse entropy function} $K_{\bs z}(\zeta): \mbb R_{\geq 0}\to\mathbb{R}_{\geq 0}\cup\{\infty\}$
    \begin{align}
    \begin{split}
        K_{\bs z}(\zeta) \coloneqq \sup\big\{\gamma\geq 0:\,  \card\big( {\mc C}_{\bs z}^\gamma\big)< e^{\zeta |\bs z|}\big\},
    \end{split}
    \end{align} 
    with the convention $\sup\varnothing=0$, and define the \emph{entropy-parameterized} path set
    \begin{align}
        {\mc H}_{\bs z}^\zeta \coloneqq {\mc C}_{\bs z}^{K_{\bs z}(\zeta)},
    \end{align}
     where ${\mc C}_{\bs z}^\infty\coloneqq\cup_{\alpha\geq 0} {\mc C}_{\bs z}^\alpha$. Sets ${\mc H}_{\bs y,\bs z}^\zeta$ for $\bs y\neq \bs o$ are defined so that the class $\mf H = \{{\mf H}^\zeta\} = \{{\mc H}^\zeta_{\bs y, \bs z}\}$ is translation invariant. 
     
As before, we denote the optimal path weights in the entropy-defined class $\mf H$ as $T({\mc H}_{\bs z}^\zeta)$. The definition is extended to real end points as in \eqref{rationalendptseqn}.
\end{definition}

In this section, our main result is the existence and the continuity (in $\zeta$ and $\bs q$) of the entropy-parametrized time constant
\begin{align}
   \mu_{{\mf H} }(\zeta,\bs q)\coloneqq \lim_{n\to\infty} \frac{T(\mc{H}^\zeta_{n\bs q})}{n},
\end{align}
provided the construction of the path class $\mf C$ satisfies some mild assumptions. Note that it is straightforward to verify that $\mf H$ is \emph{not} necessarily a filtration. Therefore our results in the previous section for filtrations are not applicable.

\subsection{Analytic Combinatorics in Several Variables}
The statements (and the proofs) of our results will involve some notions from the area of \emph{Analytic Combinatorics in Several Variables} (ACSV), which we will introduce next. (For an in-depth treatment of this field, the reader is referred to \cites{acsvbook,melczerbook}.)

The toolbox of  ACSV is geared to derive the asymptotics of the coefficients of a generating function from its analytic properties, in particular from its singularities. In the case when the generating function is rational, this can be done quite efficiently in a great variety of situations.

We need to introduce several notions first.

\begin{definition}
    The generating function $F_{\mf C}$ of the path class ${\mf C}$ is defined as
    \begin{align}
        F_{\mf C}(x, \bs z) &\coloneqq \sum_{k\in \mbb Z_{\geq 0},\,\bs r\in \mbb Z^d_{\geq 0}\setminus\{\bs o\}} \card({\mc C}_{\bs r}^{k/|\bs r|}) x^k \bs z^{\bs r}\label{countingseries}
    \end{align}
with the short-hand notation
\begin{align}
    \bs z^{\bs r} = z_1^{r_1}z_2^{r_2}\ldots z_d^{r_d}.
\end{align}
\end{definition}

As we will see, quite often this generating function is {\em rational}, i.e., can be represented as a ratio of two polynomials. We make it a standing

\begin{assumption*}
The generating function $F_{\mf C}$ can be represented as
\begin{align}
        F_{\mf C}(x, \bs z) = \frac{P(x,\bs z)}{ Q(x,\bs z)},
\end{align}
where $P,Q$ are coprime polynomials and $Q(0,\ldots,0)\neq0$.
\end{assumption*}
In particular, $F_{\mf C}$ has an open domain of absolute convergence $\mathsf D\subseteq\mbb C^{d+1}$ containing the origin where its Taylor expansion coincides with \eqref{countingseries}.

We will need the notion of {\em amoeba}:
\begin{definition}[Amoeba]
    Let $\sing_Q=\lbrace Q=0\rbrace$ be the variety defined by the polynomial $Q$. Define $\relog:\mathbb{C}_*^{d+1}\to \mathbb{R}^{d+1}$ (here $\mathbb{C}_*=\mathbb{C}-\lbrace 0\rbrace$) as
    \[
    \relog(x,\bs z)\coloneqq (\log|x|,\log|z_1|,\ldots,\log|z_d|).
    \]
    Then the \emph{amoeba} of a polynomial $Q(x,\bs z)$ is the image of $\sing_Q\cap \mathbb{C}_*^{d+1}$ under $\relog$:
    \begin{align}
        \amoeba(Q):=\relog(\sing_Q \cap \mathbb{C}_*^{d+1}).
    \end{align}

\end{definition}

The complement to amoeba $\mathbb{R}^{d+1}\setminus \amoeba(Q)$ is a disjoint union of open convex components. Different components correspond to different (converging) expansions of the rational function $P/Q$ into Laurent series (see Chapter 6 of \cite{acsvbook}). 

If $Q$ does not vanish at the origin, there is a unique component of the complement $\mathbb{R}^{d+1}\setminus \amoeba(Q)$ containing $(-M,\ldots,-M)$ for all sufficiently large $M$. We will denote it by $\mathsf B_Q$. We note that Laurent expansions corresponding to this component are the traditional expansions into Taylor series (with only coefficients with nonnegative indices).

\begin{definition}[Minimal Point]\label{defminpoint}
    The $(d+1)$-vector $(x,\bs z)\in(\mbb C_*)^{d+1}$ is called a \emph{minimal point} of $F_{\mf C}$ if $(x,\bs z)\in \partial \mathsf{D}\cap \sing_Q$ (Definition~6.34 of \cite{acsvbook}).
On nonzero-coordinate points, the domain of convergence satisfies
\[
    \mathsf D\cap(\mathbb C_*)^{d+1}=\relog^{-1}(\mathsf B_Q)
\]
(Theorems~6.4, 6.18 and Lemma~6.31 of \cite{acsvbook}).
\end{definition}

\begin{proposition}\label{minimaleqvconds}
For $F_{\mf C}$ with non-negative coefficients, fix $(x,\bs z)\in(\mbb C_*)^{d+1}$ satisfying $Q(x,\bs z)=0$. Then the following are equivalent.
    \begin{enumerate}
        \item $(x,\bs z)\in \mbb{C}^{d+1}$ is a \emph{minimal point} of $F_{\mf C}$.
        \item 
        $(|x|,|z_1|, \ldots,|z_d|)\in \mbb{R}^{d+1}$ is a \emph{minimal point} of $F_{\mf C}$.
        \item 
        $Q(t|x|,t|z_1|, \ldots,t|z_d|)\neq 0$ for all $0<t<1.$
    \end{enumerate} 
\end{proposition}
\begin{proof}
    See Proposition~6.38 (multivariate Vivanti-Pringsheim) and Corollary~6.39 of \cite{acsvbook}. 
\end{proof}
\begin{definition}\label{defsmoothpathclass}
We say that $\mf C$ is a \emph{nice} path class if its generating function $F_{\mf C}$ satisfies the following conditions:
\begin{enumerate}
    \item $F_{\mf C}$ is rational. Moreover, $Q$ is \emph{square-free}, i.e., it is equal to the product of its distinct irreducible factors.

\item For each $\bs q\in \mbb Q_{> 0}^d$, $|\bs q|=1$, there exists an open interval $\Gamma_{\bs q} \subseteq\mbb R_{>0}$, such that for all $\gamma \in \Gamma_{\bs q}$ there exists
    \begin{align}
        (u,\bs v)=(u, v_1,\ldots,v_d)\in  \mbb R_{> 0}^{d+1},
    \end{align}
    satisfying the following conditions:

\begin{enumerate}
    \item 
    $(u,\bs v)$ is a \emph{smooth critical point}
    in the direction $(\gamma, \bs q)$, i.e.,   
    \begin{align}\label{smooth cond2 eq}
        \hspace{5em}\frac{u}{\gamma}\frac{\partial Q}{\partial x} (u,\bs v) = \frac{v_1}{q_1} \frac{\partial Q}{\partial z_1} (u,\bs v) =\ldots = \frac{v_d}{q_d} \frac{\partial Q}{\partial z_d} (u,\bs v),
    \end{align}
    with $\frac{\partial Q}{\partial x} (u,\bs v) \neq 0$.
    \item $(u,\bs v)$ is a minimal point 
     of $F_{\mf C}$.
\end{enumerate}    
     
\end{enumerate}

\end{definition}
Note that in the above definition, $(u,\bs v)$ implicitly depends on $(\gamma,\bs q)$, which we omit from the notation to avoid clutter.

\subsection{Entropy--Cost Tradeoff}
\begin{theorem}\label{reparamconv}
    Assume that $\mf C$ is a nice nested feasible filtration. Let $\mf H$ be its entropy-parametrized class.  
    Then, for all $\bs q\in\mbb Q_{> 0}^d$, $|\bs q|=1$, and $\zeta\in g(\Gamma_{\bs q},\bs q)$, with $g$ as in Lemma~\ref{limlem}, the limit 
    \begin{align}
    \frac{T(\mc{H}^\zeta_{n\bs q})}{n} \limas \mu_{\mf H}(\zeta,\bs q)
    \end{align}
exists. Furthermore, the entropy-parametrized time constant $\mu_{{\mf H} }(\zeta,\bs q)$ is jointly continuous in $\zeta$ and $\bs q$ on this domain with its relative topology.  

\end{theorem}
Before the proof of Theorem above, we give some auxiliary lemmas as follows.

\begin{lemma}\label{Bsmooth}
    Assume $\mfC$ is a nice filtration with a point $(u,\bs v)\in\mathbb{R}_{>0}^{d+1}$ that satisfies the conditions-(2) of Definition~\ref{defsmoothpathclass} for a given $(\gamma,\bs q)\in\mathbb{R}_{>0}\times \mathbb{Q}_{>0}^{d}$, $|\bs q|=1$, $\gamma\in\Gamma_{\bs q}$. Then 
    $\partial \mathsf B_Q$ is a     smooth manifold in a neighborhood of $(\log u, \log \bs v)$. 
\end{lemma}

\begin{proof}
Define the function $\psi:\mathbb{R}^{d+1}\to \mathbb{R}$
   \begin{align}
       \psi(x,\bs z) \coloneqq Q(e^x,e^{\bs z})\coloneqq Q(e^x,e^{z_1}, \ldots,e^{z_d}).
   \end{align}
Partial derivatives evaluated at $(\log u, \log \bs v)$ give
\begin{align}
    \frac{\partial \psi}{\partial x} (\log u, \log \bs v) &= u \frac{\partial Q}{\partial x}(u,\bs v), \\
    \frac{\partial \psi}{\partial z_i} (\log u, \log \bs v) &= v_i \frac{\partial Q}{\partial z_i}(u,\bs v) \quad \text{for $1\leq i \leq d$}, 
\end{align}
   which are all non-zero due to  Assumption \eqref{smooth cond2 eq}. Therefore, $\nabla\psi (\log u, \log \bs v)$ is non-zero.
We prove that \emph{locally}, $\partial \mathsf B_Q = \psi^{-1}(0)$, which has two steps. First, we show $\partial \mathsf B_Q\cap \mathsf U \subseteq \psi^{-1}(0)\cap \mathsf U $ for a neighborhood $\mathsf U$ of $ (\log u, \log \bs v)$. Take a point $(\log \tilde u,\log \tilde{\bs v})\in \partial \mathsf B_Q\cap \mathsf U$ for some $(\tilde u, \tilde{\bs v})\in\mathbb{R}_{>0}^{d+1}$. Since $\mathsf B_Q$ is open $(\log \tilde u,\log \tilde{\bs v})\notin \mathsf B_Q$. Moreover, $(\log \tilde u,\log \tilde{\bs v})$ cannot be in another connected component of $\mathbb{R}^{d+1}\setminus \amoeba(Q)$ since the connected components or $\mathbb{R}^{d+1}\setminus \amoeba(Q)$ are disjoint. Therefore
$(\log \tilde u,\log \tilde{\bs v})\in\amoeba(Q)$, which means that 
\begin{align}
 Q(\tilde u e^{i\theta_0},\tilde v_1 e^{i\theta_1}, \ldots,\tilde v_d e^{i\theta_d}) = 0
\end{align}
for some $\theta_0, \theta_1,\ldots,\theta_d\in [0,2\pi)$. Assume that there exists some $\tilde\theta_0, \tilde\theta_1,\ldots,\tilde\theta_d\in [0,2\pi)$ not all zero such that $Q(\tilde u e^{i\tilde\theta_0},\tilde v_1 e^{i\tilde\theta_1}, \ldots,\tilde v_d e^{i\tilde\theta_d})=0$, but $Q(\tilde u, \tilde{\bs v})\neq 0$. Since $(\log \tilde u,\log \tilde{\bs v})\in\partial \mathsf B_Q$ and $Q(\tilde u e^{i\tilde\theta_0},\tilde{v_1}e^{i\tilde\theta_1}, \ldots,\tilde{v_d}e^{i\tilde\theta_d})=0$, $(\tilde u e^{i\tilde\theta_0},\tilde{v_1}e^{i\tilde\theta_1}, \ldots,\tilde{v_d}e^{i\tilde\theta_d})$ is a minimal point of $F_{\mathfrak C}$.
Due to Proposition~\ref{minimaleqvconds}-(2),
$(\tilde u, \tilde{\bs v})$ has to be a minimal point too, so $Q(\tilde u, \tilde{\bs v})=0$ by Definition~\ref{defminpoint}, contradicting $Q(\tilde u, \tilde{\bs v})\neq 0$.  

   To prove the reverse implication, i.e., 
   $\partial \mathsf B_Q\cap \mathsf U \supseteq \psi^{-1}(0)\cap \mathsf U$, take $(\log \tilde u,\log \tilde{\bs v}) \in \psi^{-1}(0)\cap \mathsf U$ for some $\mathsf U$ small enough such that $\nabla\psi (\log \tilde u,\log \tilde{\bs v})\neq 0$. Such $\mathsf U$ exists since $\nabla\psi (\log u, \log \bs v)$ is non-zero. Define the sets
   \begin{align}
       \mathsf U_+ &\coloneqq \mathsf U \cap \{(x,\bs z): \psi(x,\bs z)>0 \},\\
       \mathsf U_- &\coloneqq \mathsf U \cap \{(x,\bs z): \psi(x,\bs z)<0 \}.
   \end{align}
By the implicit function theorem, we may choose $\mathsf U$ convex and sufficiently small that $\mathsf U_+$ and $\mathsf U_-$ are connected.
Note that $(\log \tilde u,\log \tilde{\bs v})\notin \mathsf B_Q$ since $\psi(\log \tilde u,\log \tilde{\bs v})=0$. We will show that $(\log \tilde u,\log \tilde{\bs v})\in \partial \mathsf B_Q$ by proving that every neighborhood of $(\log \tilde u,\log \tilde{\bs v})$ includes a point $(\log \hat u,\log \hat{\bs v})$ in $\mathsf B_Q$. Previously, we had proven that $\partial \mathsf B_Q\cap \mathsf U \subseteq \psi^{-1}(0)\cap \mathsf U $, which gives $\partial \mathsf B_Q\cap \mathsf U_{+}=\varnothing$, and $\partial \mathsf B_Q\cap \mathsf U_{-}=\varnothing$. Moreover, 
\begin{align}
  \mathsf B_Q\cap \mathsf U \subseteq \mathsf U_{+} \cup \mathsf U_{-}.   
\end{align}
   Since $\mathsf B_Q\cap\mathsf U$ is convex, hence connected, and $\mathsf U_{+}$ and $\mathsf U_{-}$ are disjoint open sets, either $\mathsf B_Q\cap \mathsf U\subseteq \mathsf U_{+}$ or $\mathsf B_Q\cap \mathsf U\subseteq \mathsf U_{-}$. Assume, without loss of generality, that $\mathsf B_Q\cap \mathsf U\subseteq \mathsf U_{+}$. Now, using $\partial \mathsf B_Q\cap \mathsf U_{+}=\varnothing$, we write the disjoint union
   \begin{align}
       \mathsf U_{+} = (\mathsf U_{+}\cap \mathsf B_Q) \sqcup (\mathsf U_{+}\cap (\mathbb{R}^{d+1}\setminus (\mathsf B_Q\cup \partial \mathsf B_Q)))
   \end{align}
   We note that both $\mathsf U_{+}\cap \mathsf B_Q$ and $\mathsf U_{+}\cap (\mathbb{R}^{d+1}\setminus (\mathsf B_Q\cup \partial \mathsf B_Q))$ are open. Since $\mathsf U_{+}$ is connected, either $\mathsf U_{+}\cap \mathsf B_Q=\varnothing$ or $\mathsf U_{+}\cap (\mathbb{R}^{d+1}\setminus (\mathsf B_Q\cup \partial \mathsf B_Q))=\varnothing$. Using the assumption $\mathsf B_Q\cap \mathsf U\subseteq \mathsf U_{+}$, we obtain $\mathsf U_{+}\cap \mathsf B_Q  \supseteq \mathsf B_Q\cap \mathsf U \neq \varnothing$ because $(\log u, \log \bs v)\in \partial \mathsf B_Q$. Therefore $\mathsf U_{+}\cap (\mathbb{R}^{d+1}\setminus (\mathsf B_Q\cup \partial \mathsf B_Q))=\varnothing$, which, together with $\partial \mathsf B_Q\cap \mathsf U_{+}=\varnothing$, gives $\mathsf U_{+}\subseteq \mathsf B_Q$. Using 
   $\nabla\psi (\log \tilde u,\log \tilde{\bs v})\neq 0$ and $\psi (\log \tilde u,\log \tilde{\bs v})= 0$, we conclude that every neighborhood of $(\log \tilde u,\log \tilde{\bs v})$ includes a point in $\mathsf U_{+}$, and therefore, in $\mathsf B_Q$. This concludes the proof that $\partial \mathsf B_Q\cap \mathsf U = \psi^{-1}(0)\cap \mathsf U$ for a sufficiently small neighborhood $\mathsf U$ of $(\log u, \log\bs v)$. That 
    $\partial \mathsf B_Q$ is a     smooth manifold near $(\log u, \log \bs v)$ follows from the application of the implicit function theorem to the function $\psi$. 
\end{proof}

\begin{lemma}\label{limlem}
Let $\mf C$ be a nested feasible filtration. There is a finite, continuous, concave, one-homogeneous function
\[
g:\mbb R_{>0}^{d+1}\longrightarrow\mbb R_{\geq0}
\]
such that, for every $\bs q\in\mbb Q_{>0}^d$ with $|\bs q|=1$ and every $\gamma>0$, writing $\bs r_n=\lfloor n\bs q\rceil$,
\begin{align}\label{intrinsicentropy}
\lim_{n\to\infty}\frac{1}{n}\log\card(\mc C_{\bs r_n}^{\gamma})
=g(\gamma,\bs q).
\end{align}
If, in addition, $F_{\mf C}$ is rational, then for every $\gamma>0$ and $\bs q\in\mbb R_{>0}^{d}$,
\begin{align}\label{entropyduality}
g(\gamma,\bs q)=\inf_{\bs b\in\mathsf B_Q}(- (\gamma,\bs q)\cdot\bs b).
\end{align}
If, in addition, $\mf C$ is nice and $\gamma\in\Gamma_{\bs q}$, any positive critical point $(u_\gamma,\bs v_\gamma)$ satisfying Definition~\ref{defsmoothpathclass}(2) evaluates the rate:
\begin{align}\label{entropyheight}
g(\gamma,\bs q)=h(\gamma,\bs q)
\coloneqq-\gamma\log u_\gamma-\sum_{i=1}^dq_i\log v_{\gamma,i}.
\end{align}
\end{lemma}

\begin{proof}
For $\bs R=(k,\bs r)\in\mbb Z_{\geq0}^{d+1}$, let $A_{\bs R}=[x^k\bs z^{\bs r}]F_{\mf C}$. The filtration property and the unique split of a directed path after $|\bs r|$ steps give
\begin{align}\label{crossrayconcat}
A_{\bs R+\bs T}\geq A_{\bs R}A_{\bs T}
\end{align}
for indices with nonzero endpoints. Fix an integral $\bs R=(k,\bs r)$ with all coordinates positive. Feasibility makes $A_{m\bs R}>0$ for all sufficiently large $m$. Also $A_{m\bs R}\leq d^{m|\bs r|}$, since a directed path to $m\bs r$ has $m|\bs r|$ steps, each chosen from $d$ coordinate directions. Thus Fekete's lemma applied to $\log A_{m\bs R}$ gives the finite ray limit
\[
g(\bs R)\coloneqq\lim_{m\to\infty}m^{-1}\log A_{m\bs R},
\qquad 0\leq g(\bs R)\leq|\bs r|\log d.
\]
For $\bs R\in\mbb Q_{>0}^{d+1}$, set
\[
N\coloneqq\min\{n\in\mbb Z_{\geq1}:n\bs R\in\mbb Z^{d+1}\},
\qquad g(\bs R)\coloneqq\frac{g(N\bs R)}{N}.
\]
For integral $\bs R$ and every positive integer $j$, the subsequence $mj$ of the defining ray limit gives
\[
g(j\bs R)=j\lim_{m\to\infty}\frac{\log A_{mj\bs R}}{mj}=jg(\bs R).
\]
For rational $\bs R$, let $N_1,N_2$ be positive integers such that $N_1\bs R$ and $N_2\bs R$ are integral. Integer homogeneity gives
\[
\frac{g(N_1\bs R)}{N_1}
=\frac{g(N_1N_2\bs R)}{N_1N_2}
=\frac{g(N_2\bs R)}{N_2}.
\]
Thus the definition does not depend on the integer used to make $\bs R$ integral. For $t=J/K$ with positive integers $J,K$, choose $N$ with $N\bs R$ integral. Since $KN(t\bs R)=JN\bs R$ is integral, the same definition and integer homogeneity give
\[
g(t\bs R)=\frac{g(JN\bs R)}{KN}
=\frac{Jg(N\bs R)}{KN}=tg(\bs R).
\]
For positive rational vectors $\bs R,\bs T$, choose $N$ so that both $N\bs R$ and $N\bs T$ are integral. Apply \eqref{crossrayconcat} to $mN\bs R$ and $mN\bs T$, divide its logarithm by $mN$, and let $m\to\infty$ to obtain $g(\bs R+\bs T)\geq g(\bs R)+g(\bs T)$. Combining this inequality with rational homogeneity gives, for rational $0<t<1$,
\[
g(t\bs R+(1-t)\bs T)\geq tg(\bs R)+(1-t)g(\bs T).
\]
By this concavity and the local bound above, the standard continuity theorem for concave functions gives a unique continuous extension of $g$ to $\mbb R_{>0}^{d+1}$, preserving concavity and one-homogeneity.

We next identify this extension with the path-count limit in \eqref{intrinsicentropy}. Fix rational $\eta>0$ and $\bs q\in\mbb Q_{>0}^d$ with $|\bs q|=1$, and set
\[
N\coloneqq K_\eta\min\{n\in\mbb Z_{\geq1}:(n\eta,n\bs q)\in\mbb Z^{d+1}\},
\]
where $K_\eta$ is chosen as in Remark~\ref{axispathsrmk}, so that $\mc C^\eta_{m\bs e_i}\ne\varnothing$ for all integers $m\geq K_\eta$ and $i=1,\ldots,d$.
By \eqref{countingseries} and $|jN\bs q|=jN$,
\[
A_{(jN\eta,jN\bs q)}=\card(\mc C_{jN\bs q}^{\eta}).
\]
The definitions of $g$ on integer and rational vectors therefore give
\[
\lim_{j\to\infty}\frac{1}{jN}\log\card(\mc C_{jN\bs q}^{\eta})
=\frac{g(N\eta,N\bs q)}{N}=g(\eta,\bs q).
\]
We now pass from the subsequence $jN$ to all positive integers $n$. Recall that $\bs r_n=\lfloor n\bs q\rceil$ is the nearest lattice point to $n\bs q$, with the tie-breaking convention in \eqref{rationalendptseqn}. For $n\geq2N$, set
\[
\bs r_n^-=(\lfloor n/N\rfloor-1)N\bs q,
\qquad
\bs r_n^+=(\lceil n/N\rceil+1)N\bs q.
\]
The two vectors $\lfloor n/N\rfloor N\bs q$ and $\lceil n/N\rceil N\bs q$ are integral and lie on either side of $n\bs q$ coordinatewise. Nearest-integer rounding in each coordinate therefore gives
\[
\bs r_n^-+N\bs q\leq\bs r_n\leq\bs r_n^+-N\bs q.
\]
Thus both displacements $\bs r_n-\bs r_n^-$ and $\bs r_n^+-\bs r_n$ are coordinatewise at least $N\bs q$, whose coordinates are all at least $K_\eta$ by the choice of $N$. Remark~\ref{axispathsrmk} supplies an $\eta$-permissible path from $\bs r_n^-$ to $\bs r_n$ and one from $\bs r_n$ to $\bs r_n^+$. Appending these fixed paths gives injections and hence
\[
\card(\mc C^\eta_{\bs r_n^-})
\leq\card(\mc C^\eta_{\bs r_n})
\leq\card(\mc C^\eta_{\bs r_n^+}).
\]
Since $|\bs r_n^\pm|/n\to1$, the logarithms of both outer counts, divided by $n$, converge to $g(\eta,\bs q)$. Squeezing the middle logarithm proves \eqref{intrinsicentropy} for rational $\eta$. For real $\gamma>0$, choose rational $\eta^-<\gamma<\eta^+$ converging to $\gamma$. Nestedness sandwiches the counts at $\gamma$ between those at $\eta^\pm$; continuity of $g$ gives \eqref{intrinsicentropy}.

For $\bs b\in\mathsf B_Q$, the convergent series $M_{\bs b}=\sum_{\bs T}A_{\bs T}e^{\bs T\cdot\bs b}$ bounds each term, so $A_{mN\bs R}e^{mN\bs R\cdot\bs b}\leq M_{\bs b}$ for rational $\bs R$ and an integral multiple $N\bs R$. Taking logarithms and limits yields $g(\bs R)\leq-\bs R\cdot\bs b$. This inequality extends by continuity to every positive real $\bs R$.

Conversely, fix an arbitrary $\bs R_0\in\mbb Q_{>0}^{d+1}$ and choose a supergradient $\bs p$ of the finite concave function $g$ at $\bs R_0$:
\[
g(\bs V)\leq g(\bs R_0)+\bs p\cdot(\bs V-\bs R_0)\qquad(\bs V>0).
\]
Taking $\bs V=2\bs R_0$ and $\bs V=\bs R_0/2$ in this inequality and using one-homogeneity gives, respectively,
\[
\begin{aligned}
2g(\bs R_0)&\leq g(\bs R_0)+\bs p\cdot\bs R_0,\\
\tfrac12g(\bs R_0)&\leq g(\bs R_0)-\tfrac12\bs p\cdot\bs R_0.
\end{aligned}
\]
Subtracting $g(\bs R_0)$ from the first inequality gives $g(\bs R_0)\leq\bs p\cdot\bs R_0$. Multiplying the second inequality by $2$ and rearranging gives
\[
g(\bs R_0)\leq 2g(\bs R_0)-\bs p\cdot\bs R_0
\quad\Longrightarrow\quad
\bs p\cdot\bs R_0\leq g(\bs R_0).
\]
The two opposite inequalities therefore give $\bs p\cdot\bs R_0=g(\bs R_0)$.

For any integral index $\bs T$ with $A_{\bs T}>0$, take a positive integer $N$ clearing the denominators of $\bs R_0$, so that $N\bs R_0$ is integral. Concatenating one path counted by $A_{mN\bs R_0}$ with $mN$ paths counted by $A_{\bs T}$ gives, by repeated application of \eqref{crossrayconcat},
\[
A_{mN(\bs R_0+\bs T)}\geq A_{mN\bs R_0}\,(A_{\bs T})^{mN}.
\]
For all sufficiently large $m$, feasibility gives $A_{mN\bs R_0}>0$, so taking logarithms and dividing by $mN$ yields
\[
\frac{\log A_{mN(\bs R_0+\bs T)}}{mN}
\geq\frac{\log A_{mN\bs R_0}}{mN}+\log A_{\bs T}.
\]
By the definition of $g$, the two normalized logarithms converge to $g(N(\bs R_0+\bs T))/N=g(\bs R_0+\bs T)$ and $g(N\bs R_0)/N=g(\bs R_0)$, respectively. Hence
\[
g(\bs R_0+\bs T)\geq g(\bs R_0)+\log A_{\bs T}.
\]
On the other hand, substituting $\bs V=\bs R_0+\bs T$ in the supergradient inequality gives
\[
g(\bs R_0+\bs T)\leq g(\bs R_0)+\bs p\cdot\bs T.
\]
Combining these bounds gives $\log A_{\bs T}\leq\bs p\cdot\bs T$.

Fix $\delta>0$. If $|b_i+p_i+\delta|<\delta/2$ for every coordinate $i$, then $b_i+p_i<-\delta/2$, so
\[
\sum_{\bs T}A_{\bs T}e^{\bs T\cdot\bs b}
\leq\sum_{\bs T\in\mbb Z_{\geq0}^{d+1}}e^{-(\delta/2)|\bs T|_1}
=\bigl(1-e^{-\delta/2}\bigr)^{-(d+1)}<\infty.
\]
Thus the series converges throughout an open box about $-\bs p-\delta\bs1$, so this point lies in $\mathsf B_Q$, the logarithmic image of the open convergence domain (Definition~\ref{defminpoint}). Therefore
\[
\inf_{\bs b\in\mathsf B_Q}(-\bs R_0\cdot\bs b)
\leq g(\bs R_0)+\delta|\bs R_0|_1.
\]
Let $\delta\downarrow0$. Together with the first inequality this proves \eqref{entropyduality} for rational $\bs R_0$. Both sides are finite concave functions on the open positive cone and hence continuous, so equality extends to real $\bs R_0$.

Finally, if $\mf C$ is nice, minimality of $\bs w_\gamma=(u_\gamma,\bs v_\gamma)$ puts $\bs b_\gamma=\log\bs w_\gamma$ on $\partial\mathsf B_Q$. Write $\psi(\bs b)=Q(e^{b_0},\ldots,e^{b_d})$ for $\bs b=(b_0,\ldots,b_d)$, as in the proof of Lemma~\ref{Bsmooth}. That proof identifies the boundary near $\bs b_\gamma$ with the smooth level set $\psi=0$. The chain rule and the critical equation \eqref{smooth cond2 eq} give
\[
\begin{aligned}
\nabla\psi(\bs b_\gamma)
&=\bigl(u_\gamma Q_x(\bs w_\gamma),v_{\gamma,1}Q_{z_1}(\bs w_\gamma),\ldots,v_{\gamma,d}Q_{z_d}(\bs w_\gamma)\bigr)\\
&=\frac{u_\gamma Q_x(\bs w_\gamma)}{\gamma}\,(\gamma,\bs q)\neq0.
\end{aligned}
\]
Here \eqref{smooth cond2 eq} gives $v_{\gamma,i}Q_{z_i}(\bs w_\gamma)=q_i u_\gamma Q_x(\bs w_\gamma)/\gamma$ for $1\leq i\leq d$, and $Q_x(\bs w_\gamma)\neq0$. Differentiating $\psi=0$ along any smooth curve on the boundary shows that its tangent vector is perpendicular to $\nabla\psi(\bs b_\gamma)$ at $\bs b_\gamma$. Thus $(\gamma,\bs q)$ is also normal to the boundary, and the tangent hyperplane is
\[
\{\bs b:(\gamma,\bs q)\cdot(\bs b-\bs b_\gamma)=0\}.
\]
Since $\mathsf B_Q$ is convex, this tangent hyperplane supports it: the entire domain lies on one side of the hyperplane. Downward closure gives $\bs b_\gamma-\varepsilon\bs1\in\mathsf B_Q$ for every $\varepsilon>0$; since $(\gamma,\bs q)$ has positive coordinates, that side is
\[
(\gamma,\bs q)\cdot(\bs b-\bs b_\gamma)\leq0
\qquad(\bs b\in\overline{\mathsf B}_Q).
\]
The interior points $\bs b_\gamma-\varepsilon\bs1$ approach $\bs b_\gamma$ as $\varepsilon\downarrow0$. Hence the supporting inequality and \eqref{entropyduality} give
\[
\begin{aligned}
g(\gamma,\bs q)
&=\inf_{\bs b\in\mathsf B_Q}[-(\gamma,\bs q)\cdot\bs b]
=-(\gamma,\bs q)\cdot\bs b_\gamma\\
&=-\gamma\log u_\gamma-\sum_{i=1}^d q_i\log v_{\gamma,i}
=h(\gamma,\bs q),
\end{aligned}
\]
proving \eqref{entropyheight}.
\end{proof}

\begin{lemma}\label{strictentropy}
    If $\mf C$ is a nice nested feasible filtration, then the function $g(\gamma,\bs q)$, as defined in Lemma~\ref{limlem}, is a strictly concave, strictly increasing function of $\gamma\in\Gamma_{\bs q}$. 
\end{lemma}

\begin{proof}
By Lemma~\ref{limlem}, the minimum is achieved at a positive critical point, i.e., 
\begin{align}
    g(\gamma,\bs q) &= \min_{(\bar{u},\bar{\bs v})\in \bar {\mathsf B}_Q}  (- \gamma\bar u    
  -  \bs q\cdot  \bar{\bs v} )  = -\max_{(\bar{u},\bar{\bs v})\in \bar {\mathsf B}_Q}  (\gamma\bar u +    
   \bs q\cdot  \bar{\bs v} ),
\end{align}
where $\bar {\mathsf B}_Q$ is the closure of the set $\mathsf B_Q$.
 Therefore, $g(\gamma,\bs q)$ is a concave function of $\gamma$ when $\gamma\in\Gamma_{\bs q}$. We will indeed prove that it is \emph{strictly} concave. Assume there exists $\gamma_1,\gamma_2\in \Gamma_{\bs q}$, $\gamma_1\neq \gamma_2$ and $0<\lambda<1$ such that
\begin{equation*}
 \max_{(\bar{u},\bar{\bs v})\in \bar{\mathsf B}_Q} ( \lambda\gamma_1\bar u + (1-\lambda)\gamma_2\bar u+    
   \bs q\cdot  \bar{\bs v}) = \lambda\max_{(\bar{u},\bar{\bs v})\in \bar{\mathsf B}_Q} ( \gamma_1\bar u +    
   \bs q\cdot  \bar{\bs v} )+ (1-\lambda)\max_{(\bar{u},\bar{\bs v})\in\bar{\mathsf B}_Q} ( \gamma_2\bar u +    
   \bs q\cdot  \bar{\bs v} ). 
\end{equation*}
Let $\gamma_*=\lambda\gamma_1+(1-\lambda)\gamma_2$ and $(\bar u_*,\bar{\bs v}_*)=(\log u_{\gamma_*},\log\bs v_{\gamma_*})$, a maximizer of the expression on the left hand side. Using the inequalities
\begin{align}
     \lambda\gamma_1\bar u_* +    
    \lambda\bs q\cdot \bar{\bs v}_* &\leq \lambda\max_{(\bar{u},\bar{\bs v})\in \bar{\mathsf B}_Q} ( \gamma_1\bar u +    
   \bs q\cdot  \bar{\bs v} ),\quad\textrm{and}\\  
    (1-\lambda)\gamma_2\bar u_* +   
   (1-\lambda)\bs q\cdot  \bar{\bs v}_* &\leq (1-\lambda)\max_{(\bar{u},\bar{\bs v})\in\bar{\mathsf B}_Q} ( \gamma_2\bar u +    
   \bs q\cdot  \bar{\bs v} ),
\end{align}
and the previous equality, we obtain 
\begin{align}
    \gamma_1\bar u_* + \bs q\cdot \bar{\bs v}_* &= \max_{(\bar{u},\bar{\bs v})\in \bar{\mathsf B}_Q} ( \gamma_1\bar u +    
   \bs q\cdot  \bar{\bs v} ) = -g(\gamma_1,\bs q)\\
   \gamma_2\bar u_* + \bs q\cdot \bar{\bs v}_* &= \max_{(\bar{u},\bar{\bs v})\in \bar{\mathsf B}_Q} ( \gamma_2\bar u +    
   \bs q\cdot  \bar{\bs v} ) = -g(\gamma_2,\bs q)
\end{align}
Note that both 
\begin{align*}
    \{ (\bar{u},\bar{\bs v})\in\mathbb{R}^{d+1} : \gamma_1\bar u +    
   \bs q\cdot  \bar{\bs v} &= -g(\gamma_1,\bs q) \} \\
   \{ (\bar{u},\bar{\bs v})\in\mathbb{R}^{d+1} : \gamma_2\bar u +    
   \bs q\cdot  \bar{\bs v} &= -g(\gamma_2,\bs q) \}
\end{align*}
define different supporting hyperplanes to $\bar{\mathsf B}_Q$ that go through $(\bar u_*, \bar{\bs v}_*)$. This contradicts the smoothness of $\partial\mathsf B_Q$, i.e., Lemma~\ref{Bsmooth}.

Finally, note that due to the nestedness of $\mfC$, by definition, $g(\gamma,\bs q)$ is non-decreasing in $\gamma$. Strict concavity gives that it is strictly increasing in $\gamma$.

\end{proof}

\begin{lemma}\label{hcequivlem}
Let $\mf C$ be a nice nested feasible filtration and $(\gamma,\bs q)\in\mathbb{R}_{>0}\times \mathbb{Q}_{>0}^{d}$, $|\bs q|=1$, $\gamma\in\Gamma_{\bs q}$, as in Definition~\ref{defsmoothpathclass}. Let $\mf{H}$ be the entropy-parameterization of $\mf C$, Definition ~\ref{entropdef}. The function $g$, defined in Lemma~\ref{limlem}, gives the parameter exchange between the time constants of $\mf H$ and $\mf C$, i.e., if $\zeta=g(\gamma,\bs q)$ then 
\begin{align}\label{exchangeeq}
 \lim_{n\to\infty} \frac{T(\mc{H}^\zeta_{n\bs q})}{n}\eqqcolon\mu_{{\mf H} }(\zeta,\bs q)=\mu_{{\mf C} }(\gamma,\bs q).
\end{align}
\end{lemma}
\begin{proof}
Note that since $g(\cdot,\bs q)$ is a strictly increasing continuous function, it has a continuous inverse $g^{-1}(\cdot,\bs q)$.
 Existence of the limit $\mu_{{\mf H} }(\zeta,\bs q)$ and the identity \eqref{exchangeeq} follows by Lemma~\ref{limlem}, nestedness, and $|\lfloor n\bs q\rceil|/n\to1$: for sufficiently small $\epsilon>0$, and sufficiently large $n$,
\begin{align}
    \frac{T(\mc C^{g^{-1}(\zeta + \epsilon,\bs q)}_{n\bs q})}{n}  \leq \frac{T(\mc H^\zeta_{n\bs q})}{n} \leq \frac{T(\mc C^{g^{-1}(\zeta - \epsilon,\bs q)}_{n\bs q})}{n},
\end{align}
and letting $\epsilon\downarrow0$ while invoking Lemma~\ref{subadd}, the continuity of $\mu_{{\mf C}}(\gamma,\bs q)$ in $\gamma$, and the continuity of $g^{-1}(\zeta,\bs q)$ in $\zeta$. 
\end{proof}

\begin{proof}[Proof of Theorem~\ref{reparamconv}]
We first verify continuity of $g^{-1}$. Let $(\zeta_j,\bs q_j)\to(\zeta,\bs q)$ in the stated domain and put $\gamma_j=g^{-1}(\zeta_j,\bs q_j)$, $\gamma=g^{-1}(\zeta,\bs q)$. For $a<\gamma<b$ in $\Gamma_{\bs q}$, Lemma~\ref{limlem} and strict increase give $g(a,\bs q_j)<\zeta_j<g(b,\bs q_j)$ for large $j$. Nestedness implies $a<\gamma_j<b$; letting $a,b\to\gamma$ proves $\gamma_j\to\gamma$.

The conclusion follows by the identity $\mu_{{\mf H} }(\zeta,\bs q)=\mu_{{\mf C}}(g^{-1}(\zeta,\bs q),\bs q)$ due to Lemma~\ref{hcequivlem}, and the joint continuity of $\mu_{{\mf C}}$ due to Lemma~\ref{convexitylem}.
\end{proof}

\section{Examples}\seceq
\subsection{Constrained by Number of Turns}\label{turnssmooth}
Consider the situation of Example~\ref{numoftruns exa}.
For simplicity, assume for this example that $d=2$. 

To find an algebraic expression for $F_{\mf C}$ for this class, we will first consider the generating function
\begin{align*}
    G(x,y,z)    
    \coloneqq \hspace{-1em}\sum_{k,\ell,m\in \mbb Z_{\geq 0}} \{ \text{\# of paths $\mb A$ from $\bs o$ to $(\ell,m)$,  $f(\mb A)=k$}\}  x^k y^\ell z^m
\end{align*}
of the \emph{level sets}, where $f$ counts the number of turns, and note that 
\begin{align*}
    F_{\mf C}(x,y,z) = \hspace{-1em}\sum_{k,\ell,m\in \mbb Z_{\geq 0}} \{ \text{\# paths $\mb A$ from $\bs o$ to $(\ell,m)$,  $f(\mb A)\leq k-1$}\}  x^k y^\ell z^m,
\end{align*}
and therefore the relationship
\begin{align*}
    F_{\mf C}(x,y,z) = \frac{xG(x,y,z)}{1-x}. 
\end{align*}

\begin{theorem}
The generating function is equal to
\begin{align}
    F_{\mf C} = \frac{x(-2xyz+2yz-y-z)}{(x^2yz - yz + y+z -1)(1-x)}\eqqcolon \frac{P(x,y,z)}{Q(x,y,z)}.
\end{align}
Moreover, $\mf C$ is a nice path class, with
\[
    \Gamma_{\bs q}=(0,2q_1q_2)
\]
for every $\bs q=(q_1,q_2)\in\mbb Q_{>0}^{2}$ satisfying $q_1+q_2=1$.
\end{theorem}

To prove it, we define the generating function  
\begin{align*}
\begin{alignedat}{2}
    &G^\EE(x,y,z)    
    \coloneqq \sum_{k,\ell,m\in \mbb Z_{\geq 0}} \big\{ \text{\# of paths} \\&\text{$\mb A=\big(\bs o, (1,0), \bs a_2, \ldots, \bs a_{|\mb A|-2},(\ell-1,m), (\ell,m)\big) $ ,  $f(\mb A)=k$}\big\}  x^k y^\ell z^m,
\end{alignedat}
\end{align*}
which counts the paths of length at least two with fixed first and last steps (both in the East direction, i.e. (1,0)). Generating functions $G^\NE$, $G^\EN$, $G^\NN$ are similarly defined. For instance, $G^\NE$ counts the paths in the form $\mb A = \big(\bs o, \bs a_1,  \ldots, \bs a_{|\mb A|}\big)$ with $\bs a_1 = (0,1)$, $\bs a_{|\mb A|-1} = (\ell-1,m)$, and $\bs a_{|\mb A|}=(\ell,m)$. Note  that the following two identities hold.
\begin{align}
G = G^\EE+ G^\NE + G^\EN + G^\NN + y+ z,
\end{align}
and 
\begin{align}
G^\EE = x^2y^2G^\NN + xy^2G^\EN + y^2G^\EE + xy^2G^\NE + y^2 +y^3 +x^2y^2z 
\end{align}
 The first is due to that $G$ counts the paths with all possible start and end steps, and in addition it includes the paths of size one: $(\bs o, (0,1))$ and $(\bs o, (1,0))$. The second identity follows the bijection between the set of paths with first and last steps towards East and the set of all paths with East step concatenated in the beginning and at the end. Additional $x$ terms in front of the $G^{\mathrm{XX}}$, $\mathrm{X}\in \{\mathrm{E},\mathrm{N}\}$, represent the additional turns at the concatenations. Since $G^{\mathrm{XX}}$ counts only the paths of length at least two, the terms $y^2$ and $y^3$ are added to account for the paths of length two and three with all East steps, while $x^2y^2z$ is for the path East-North-East.  
 Similarly, we have,
\begin{align*}
   G^\NN &= z^2 G^\NN + xz^2G^\EN + x^2z^2G^\EE + xz^2G^\NE + z^2 +z^3 +x^2yz^2  \\
   G^\EN &= xyz G^\NN + yzG^\EN + xyz G^\EE + x^2yz G^\NE + xyz + xy^2z +xyz^2\\
   G^\NE &= xyz G^\NN + x^2yzG^\EN + xyz G^\EE + yz G^\NE + xyz + xy^2z +xyz^2
\end{align*}
 
 Solving these identities for $G$ we obtain
\begin{align}\label{eq:G}
    G = \frac{-2xyz+2yz-y-z}{x^2yz - yz + y+z -1}.
\end{align}
Together with $F_{\mf C}=xG/(1-x)$, this proves the formula and supplies the rational representation in Condition~(1) of Definition~\ref{defsmoothpathclass}.

Figure~\ref{fig:example-convergence}(a) shows the convergence region $\mathsf B_Q$.

 Next we verify that $F_{\mf C}$ satisfies the Condition~(2) of Definition~\ref{defsmoothpathclass}. Note first that $Q$ is square-free in this example.
 For some $\gamma\in\mbb R_{>0}$, $\bs q= (q_1,q_2)\in \mbb Q_{>0}^2$, $q_1+q_2=1$, and a candidate minimal point $(u,v_1,v_2)$, Equation~\refeq{smooth cond2 eq} reads 
\begin{align}
    \frac{u}{\gamma}\frac{\partial Q}{\partial x}(u,v_1,v_2) = \frac{v_1}{q_1} \frac{\partial Q}{\partial y}(u,v_1,v_2) = \frac{v_2}{q_2} \frac{\partial Q}{\partial z} (u,v_1,v_2)\neq 0,
\end{align}
 which is equivalent to
\begin{align}
    \frac{u}{\gamma}&[2uv_1v_2(1-u) - (u^2v_1v_2-v_1v_2+v_1+v_2-1) ] \\&= \frac{v_1}{q_1} (1-u) (u^2v_2-v_2+1)
    = \frac{v_2}{q_2} (1-u) (u^2v_1-v_1+1),
\end{align}
$u\neq 1$. In addition, the singularity requirement of Condition~2-(b) gives 
\begin{align}
    Q(u,v_1,v_2) = (1-u)(u^2v_1v_2-v_1v_2+v_1+v_2-1) =  0.
\end{align}
It is straightforward to verify that these conditions imply
\begin{align}
    u &= \frac{\gamma}{\sqrt{(2q_1-\gamma)(2q_2-\gamma)}},\\
    v_1 &= 1-\frac{\gamma}{2q_1}, \quad
    v_2 = 1-\frac{\gamma}{2q_2}.
\end{align}
 Choosing $\gamma\in (0,2q_1q_2)$ ensures that $0<u<1$ and $v_1,v_2>0$. 
Next, we check the minimality of $(u,v_1,v_2)$. Through some algebra, we find that  
\begin{align*}
\begin{alignedat}{2}
    Q(tu,tv_1,tv_2) = \frac{(1-tu)(t-1)}{4q_1q_2} \big[ &\gamma^2(t+1)(t^2+1) + 2\gamma(1-\gamma)\\& + (1-t)(2q_1-\gamma)(2q_2-\gamma)\big]. 
\end{alignedat}
\end{align*} 
For $0<t<1$, the factors $1-tu$, $\gamma^2(t+1)(t^2+1)$, $2\gamma(1-\gamma)$, and $(1-t)(2q_1-\gamma)(2q_2-\gamma)$ are all strictly positive, while $t-1\neq0$.  Hence $Q(tu,tv_1,tv_2)\neq0$ for $0<t<1$, which verifies the minimality of $(u,v_1,v_2)$ by Proposition~\ref{minimaleqvconds}.
 
Consequently, we have proved that $\mf C$ associated with the path function that counts the number of turns is indeed a nice path class with the choice
\[
    \Gamma_{\bs q}\coloneqq (0,2q_1q_2).
\]

\subsection{Consecutive Runs}\label{2Es}
    For this example, let $f$ count the number of sequences of consecutive East steps of length at least two, that is, for a path $\mb A = \big(\bs a_0,  \ldots, \bs a_{|\mb A|}\big)$ with $|\mb A|\geq2$,
    \begin{align*}
    \begin{alignedat}{2}
        &f(\mb A)\coloneqq 
        \ind\{\bs a_{2}-\bs a_{1} = \bs a_{1}-\bs a_{0} = (1,0)\}+\\
        &\sum_{i=3}^{|\mb A|}\ind\{\bs a_{i}-\bs a_{i-1} = \bs a_{i-1}-\bs a_{i-2} = (1,0)\}
        \ind\{\bs a_{i-2}-\bs a_{i-3} = (0,1)\}.
    \end{alignedat}
    \end{align*}
For $|\mb A|=1$, set $f(\mb A)=0$.
Define the path class $\mf C$ accordingly, as in \eqref{Czgammadefeqn}.
\begin{theorem}
The generating function is equal to
\begin{align}
 F_{\mf C}(x,y,z) = \frac{x(-xy^2z - xy^2 + y^2z + y^2 - y - z)}{(xy^2z - y^2z + y + z - 1)(1-x)}\eqqcolon\frac{P(x,y,z)}{Q(x,y,z)}.
\end{align}
Moreover, $\mf C$ is a nice path class. For every $q\in\mbb Q\cap(0,1)$, there exists $\iota(q)>0$ such that, with $\bs q=(q,1-q)$, one may take
\[
\Gamma_{\bs q}=(0,\iota(q)).
\]
\end{theorem}

To prove it, we first write the relations
\begin{align}
    G^\E(x,y,z) &= y + yG^\N(x,y,z) \\
    G^\EE(x,y,z) &= xyG^\E(x,y,z) + yG^\EE(x,y,z) \\
    G^\N(x,y,z) &=  z + zG^\E(x,y,z) + zG^\EE(x,y,z) + zG^\N(x,y,z) 
\end{align}
where $G^\E$, $G^\EE$, and $G^\N$ are the generating functions of paths whose final East-run has length one, whose final East-run has length at least two, and which end with a North step, respectively. Here $x$ marks the number of East-runs of length at least two, and $y,z$ mark East and North steps.
Note that 
\[F_{\mf C} =  \frac{x(G^\E+G^\EE+G^\N)}{1-x}.
\] 
Solving these equations and substituting in this identity gives the stated formula.

Figure~\ref{fig:example-convergence}(b) shows the convergence region $\mathsf B_Q$.

\begin{figure}[tbp]
    \centering
    \includegraphics[width=\textwidth]{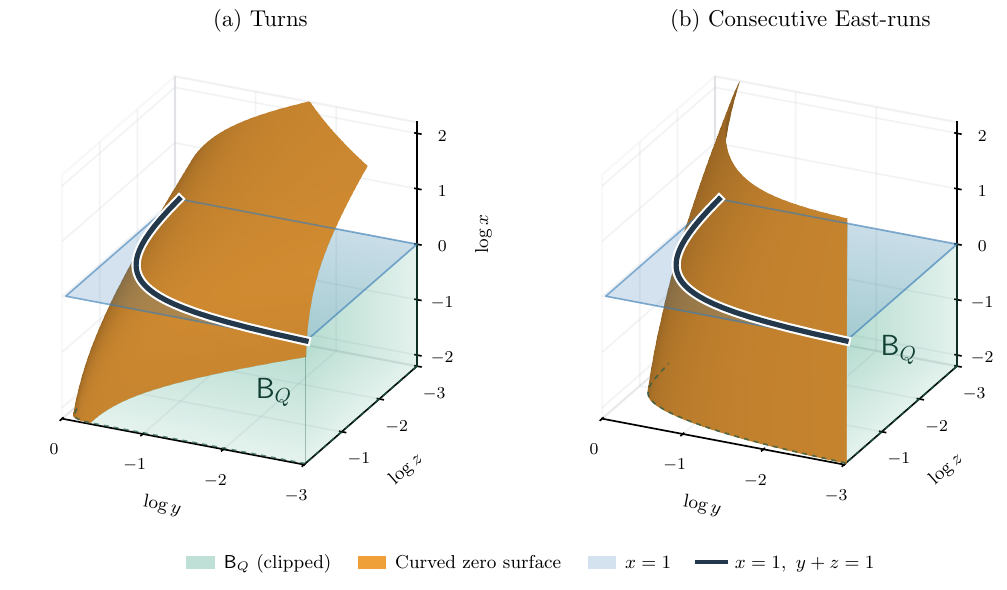}
    \caption{Convergence regions for turns (left) and consecutive East-runs (right), on identical logarithmic scales with $0<y,z<1$. Green shows the portion of the open region $\mathsf B_Q$ within the plotting window. The positive-real zero surfaces of $Q$ are $x=1$ (blue) and, respectively, $x^2yz=(1-y)(1-z)$ and $xy^2z=1-y-z+y^2z$ (orange); the dark curve is their common intersection $x=1$, $y+z=1$. For every fixed $0<y,z<1$, sufficiently small $x>0$ gives convergence on the left, whereas on the right this holds exactly when $z(1+y)<1$. This difference reflects the paths with zero counted events: zero turns permits only straight paths, while zero counted East-runs permits every path with no consecutive East steps.}
    \label{fig:example-convergence}
\end{figure}
 Again, $Q$ is square-free in this example. For a direction $(\gamma, q_1,q_2)$ and a candidate minimal point $(u,v_1,v_2)$, smooth critical point condition \eqref{smooth cond2 eq} reads as
\begin{align}\label{criteenneq}
\begin{alignedat}{1}
    &\frac{u}{\gamma}[(1-u)v_1^2v_2  - (uv_1^2v_2 - v_1^2v_2  + v_1 + v_2 -1) ] \\&= \frac{v_1}{q_1} (1-u)(2uv_1v_2 - 2v_1v_2  +1)= \frac{v_2}{q_2} (1-u) (uv_1^2 - v_1^2 + 1 ),
    \end{alignedat}
\end{align}
with $u\neq 1$. Along with the pole condition
\begin{align}\label{poleeq}
    uv_1^2v_2 - v_1^2v_2  + v_1 + v_2 -1 = 0,
\end{align}
\eqref{criteenneq} reads as
\begin{align}\label{criteeneqqsimp}
\begin{alignedat}{1}
    &\frac{uv_1^2v_2}{\gamma}= \frac{1-v_1}{1-q} = \frac{2-v_1-2v_2}{q}  ,
    \end{alignedat}
\end{align}
where we assumed that $\bs q = (q,1-q)$. Solving $v_2$ in terms of $v_1$ using the second equation in \eqref{criteeneqqsimp}, we obtain
\begin{equation}\label{v2inv1ueq}
    v_2 = \frac{2-3q+(2q-1)v_1}{2(1-q)}.
\end{equation}
Combining \eqref{poleeq} and the first equation of \eqref{criteeneqqsimp} in order to cancel out the higher order terms, we get
\begin{equation}
    \frac{\gamma (u-1)(1-v_1) + u(1-q)(v_1+v_2-1) }{u(1-q)} = 0.
\end{equation}
 Plugging \eqref{v2inv1ueq} in the equation above implies
\begin{equation}
    \gamma uv_1 -\gamma u -\gamma v_1 +\gamma +\frac{qu}{2} - \frac{uv_1}{2} =0,
\end{equation}
for $u\neq 0$, $q\neq 1$, which lets us write $u$ in terms of $v_1$ as
\begin{equation}\label{u_in_v1_eqn}
    u = \frac{2\gamma (v_1-1)}{2\gamma (v_1-1) + q -v_1}.
\end{equation}
Plugging the expressions for $v_2$ and $u$ in terms of $v_1$, the first equation of \eqref{criteeneqqsimp} takes the form
\begin{equation}
    \frac{\left(v_{1} - 1\right)^{2} \left(- 2 \gamma - 2 q v_{1}^{2} + q v_{1} + q + v_{1}^{2} - v_{1}\right)}{\left(q - 1\right) \left(2 \gamma v_{1} - 2 \gamma + q - v_{1}\right)} = 0.
\end{equation}
As $v_1=1$ would make $u=0$, and we are interested in only the positive solutions $(u,v_1,v_2)$, the equation above implies
\begin{equation}\label{v1eqnquad}
    (1-2q)v_1^2 - (1-q)v_1 + q-2\gamma = 0,
\end{equation}
which is a quadratic polynomial in $v_1$ if $q\neq 1/2$, in which case, it has the solutions
\begin{equation}\label{v1sol_eqn}
    v_1 = \frac{(1-q)\pm\sqrt{(3q-1)^2 - 8(2q-1)\gamma}}{2(1-2q)}. 
\end{equation}
We denote the solution with a positive sign in front of the square root in the expression above with $v_1^+$, and $v_1^-$ with the negative sign. Also, given \eqref{u_in_v1_eqn} and \eqref{v2inv1ueq}, let 
\begin{align}\label{uv2soleqn}
\begin{alignedat}{2}
    u^+ &\coloneqq \frac{2\gamma (v_1^+-1)}{2\gamma (v_1^+-1) + q -v_1^+}, \quad && u^- \coloneqq \frac{2\gamma (v_1^- -1)}{2\gamma (v_1^- -1) + q -v_1^-}, \\
    v_2^+ &\coloneqq \frac{2-3q+(2q-1)v_1^+}{2(1-q)},\quad && v_2^- \coloneqq \frac{2-3q+(2q-1)v_1^-}{2(1-q)}.
    \end{alignedat}
\end{align}

For the solutions to be real we need 
$ 8(2q-1)\gamma \leq (3q-1)^2. 
$
Let us first consider the case $0<q<1/2$, which guarantees that $v_1^+$ and $v_1^-$ are real. 
We find, through a symbolic computation tool, that
\begin{align}\label{u-u+eqn}
\begin{alignedat}{1}
    u^-  u^+ &= \frac{4\gamma^3}{(-2\gamma + q)^2 (\gamma +q-1)} < 0 \\
    u^- - u^+ &= \frac{\gamma (1-q) \sqrt{(3q-1)^2 - 8(2q-1)\gamma}}{ (-2\gamma +q)^2 (1-q-\gamma)} > 0
    \end{alignedat}
\end{align}
for $0<\gamma<q/2<1/2$. Therefore $u^->0$ and $u^+<0$. Furthermore, we find 
\begin{align}
    v_2^- &=\frac{3-5q + \sqrt{(3q-1)^2 - 8(2q-1)\gamma}}{4(1-q)}>0 
\end{align}
for $0<\gamma<q/2<1/2$, and 
\begin{align}
    1-q > \sqrt{(3q-1)^2 - 8(2q-1)\gamma}
\end{align}
when $\gamma<q/2$ giving $v_1^->0$ due to \eqref{v1sol_eqn}. We conclude that $(u^-, v_1^-, v_2^-)$ is the unique smooth critical point with positive real entries, provided $u^-\neq1$.

Let us now consider the $q=1/2$ case, for which, equation \eqref{v1eqnquad} gives $v_1=1-4\gamma$, and \eqref{u_in_v1_eqn} and \eqref{v2inv1ueq} give 
\begin{align}
    u&= \frac{\gamma^2}{(1/4 - \gamma )^2},\quad 
    v_2 = 0.5,
\end{align}
which are positive for $\gamma<1/4$.

For $1/2<q<1$, note that the solution $u^-$ given in \eqref{uv2soleqn} 
is still the unique positive real solution, due to \eqref{u-u+eqn}, to smooth critical pole equations for $u$ provided $\gamma\neq q/2$, $u^-\neq1$, and
\begin{align}
    \gamma<\min\left\{1-q, \frac{(3q-1)^2}{8(2q-1)}\right\}.
\end{align}
Furthermore, $v_1^-$ and $v_2^-$ are positive if 
\begin{align}
      \sqrt{(3q-1)^2 - 8(2q-1)\gamma}> \max\{ 1-q, 5q-3\},
\end{align}
which holds if $\gamma<q/2$ when $1/2<q\leq2/3$ and if $\gamma<(1-q)/2$ when $2/3<q<1$. Thus $(u^-, v_1^-,v_2^-)$ is the unique positive real solution to smooth critical pole conditions when $u^-\neq1$ and $\gamma<\min\left\{\frac{1-q}{2}, \frac{(3q-1)^2}{8(2q-1)}\right\}$. A straightforward calculation reveals that $\frac{1-q}{2}$ is always less than  $\frac{(3q-1)^2}{8(2q-1)}$, therefore 
the condition can be simplified to $\gamma<(1-q)/2$.

To check minimality for all $0<q<1$, denote the positive solution above by $(u,v_1,v_2)$. The relevant root of \eqref{v1eqnquad} lies in $(0,1)$ for $0<\gamma<q/2$, since the polynomial is positive at $0$ and negative at $1$. Moreover, as $\gamma\downarrow0$, the explicit formulas give
\[
    (v_1,v_2)\longrightarrow
    \begin{cases}
        \left(\dfrac{q}{1-2q},\dfrac{1-2q}{1-q}\right),&0<q<1/3,\\[6pt]
        (1,1/2),&1/3\leq q<1.
    \end{cases}
\]
By \eqref{criteeneqqsimp}, $u=\gamma(1-v_1)/[(1-q)v_1^2v_2]\to0$. Thus, for each fixed $q$, we may choose $\iota(q)>0$ so that $0<u,v_1<1$ and $v_2>0$ whenever $0<\gamma<\iota(q)$. The factorization
\[
    Q(x,y,z)=-(1-x)(1-y)\left[1-z\left(1+y+\frac{xy^2}{1-y}\right)\right]
\]
and $Q(u,v_1,v_2)=0$ imply $v_2(1+v_1+uv_1^2/(1-v_1))=1$. Hence, for $0<t<1$,
\[
    tv_2\left(1+tv_1+\frac{t^3uv_1^2}{1-tv_1}\right)
    <v_2\left(1+v_1+\frac{uv_1^2}{1-v_1}\right)=1.
\]
Since $1-tu$ and $1-tv_1$ are positive, $Q(tu,tv_1,tv_2)<0$, proving minimality by Proposition~\ref{minimaleqvconds}. Also $Q_x(u,v_1,v_2)=(1-u)v_1^2v_2>0$. Consequently, the path class is nice for $\Gamma_{\bs q}=(0,\iota(q))$.

\section{Conclusion}
We have introduced a model of the efficiency--complexity tradeoff in which an agent's physical cost is the minimum total edge cost within a constrained family of directed paths, and representational complexity is measured by the growth of that family. The path-joining property provides the structure needed to establish deterministic large-distance costs and their convexity, together with a concave path-count growth rate, under the stated assumptions. For classes satisfying the additional generating-function conditions, this rate permits an entropy parameterization of the limiting cost, which is jointly continuous in entropy and direction on the corresponding domain. The two planar examples show how restrictions on local path features fit this framework.

We view this construction as a baseline for a broader class of models coupling physical performance to limited representational resources, with potential applications in the natural sciences, engineering, and the theory of stochastic processes. A natural next problem is to quantify the cost penalty imposed by path restrictions, including how this penalty decreases as turn or run constraints are relaxed, and how it depends on direction and the edge-weight distribution. More broadly, the formulation could be extended by incorporating more realistic biological constraints, giving rise to a wider family of models and potential applications.

\section*{Acknowledgments}
We thank Daniel E. Koditschek for insightful discussions that helped shape the problem formulation and its interpretation in terms of resource-limited agents.

\section*{AI tool disclosure}
The second author used OpenAI's large language models through Codex to assist with reviewing mathematical arguments, developing and simplifying proofs, checking calculations, and revising the exposition and \LaTeX{} source. AI-generated suggestions were reviewed and revised by the second author. The authors take responsibility for the mathematical content and final text.

%% file: tikz/BranchRW.tex
\begin{tikzpicture}[
    x=1cm,y=1cm,
    every node/.style={font=\small},
    vertex/.style={circle, draw, inner sep=1.4pt, minimum size=15pt},
    minvertex/.style={circle, draw, fill=red!15, inner sep=1.4pt, minimum size=15pt},
    edge/.style={draw, thick},
    minedge/.style={draw, very thick, red},
    edgelab/.style={midway, fill=white, inner sep=1pt, font=\scriptsize},
    genline/.style={dashed, gray}
]

\node[vertex] (r)   at (0,0) {};

\node[minvertex] (a)   at (2, 2.4) {};
\node[vertex] (b)   at (2,-2.4) {};

\node[vertex] (aa)  at (4, 3.6) {};
\node[minvertex] (ab)  at (4, 1.2) {};
\node[vertex] (ba)  at (4,-1.2) {};
\node[vertex] (bb)  at (4,-3.6) {};

\node[vertex] (aaa) at (6, 4.2) {};
\node[vertex]    (aab) at (6, 3.0) {};
\node[minvertex]    (aba) at (6, 1.8) {};
\node[vertex]    (abb) at (6, 0.6) {};
\node[vertex]    (baa) at (6,-0.6) {};
\node[vertex]    (bab) at (6,-1.8) {};
\node[vertex]    (bba) at (6,-3.0) {};
\node[vertex]    (bbb) at (6,-4.2) {};

\draw[minedge] (r) -- (a) node[edgelab, above left] {$\tau_{1,0}$};
\draw[edge]    (r) -- (b) node[edgelab, below left]{$\tau_{1,1}$};

\draw[edge] (a) -- (aa) node[edgelab, above] {$\tau_{2,0}$};
\draw[minedge]    (a) -- (ab) node[edgelab, below] {$\tau_{2,1}$};
\draw[edge]    (b) -- (ba) node[edgelab, above] {$\tau_{2,2}$};
\draw[edge]    (b) -- (bb) node[edgelab, below] {$\tau_{2,3}$};

\draw[edge] (aa) -- (aaa) node[edgelab, above] {$\tau_{3,0}$};
\draw[edge]    (aa) -- (aab) node[edgelab, below] {$\tau_{3,1}$};
\draw[minedge]    (ab) -- (aba) node[edgelab, above] {$\tau_{3,2}$};
\draw[edge]    (ab) -- (abb) node[edgelab, below] {$\tau_{3,3}$};
\draw[edge]    (ba) -- (baa) node[edgelab, above] {$\tau_{3,4}$};
\draw[edge]    (ba) -- (bab) node[edgelab, below] {$\tau_{3,5}$};
\draw[edge]    (bb) -- (bba) node[edgelab, above] {$\tau_{3,6}$};
\draw[edge]    (bb) -- (bbb) node[edgelab, below] {$\tau_{3,7}$};

\draw[genline] (0,-4.8) -- (0,4.8);
\draw[genline] (2,-4.8) -- (2,4.8);
\draw[genline] (4,-4.8) -- (4,4.8);
\draw[genline] (6,-4.8) -- (6,4.8);

\node[below] at (0,-4.9) {layer-\(0\)};
\node[below] at (2,-4.9) {layer-\(1\)};
\node[below] at (4,-4.9) {layer-\(2\)};
\node[below] at (6,-4.9) {layer-\(n\)};

\node[align=left, anchor=west, red!80!black] at (6.8,1.8) {%
\(\displaystyle M_n=\tau_{1,0}+\tau_{2,1}+\tau_{3,2}\)
};

\end{tikzpicture}